\documentclass[journal ]{new-aiaa}
\usepackage[utf8]{inputenc}
\usepackage{textcomp}

\usepackage{graphics} 
\usepackage{graphicx}
\usepackage{epsfig} 
\usepackage{amsmath} 

\usepackage{amsthm}
\usepackage{subcaption}
\usepackage{xcolor}
\usepackage{float}


\newcounter{parentnumber}

\newtheorem{theorem}{Theorem}
\newtheorem{lemma}{Lemma}
\newtheorem{definition}{Definition}
\newtheorem{remark}{Remark}
\newtheorem{assumption}{Assumption}

\usepackage[version=4]{mhchem}
\usepackage{siunitx}
\usepackage{longtable,tabularx}
\title{Continuous Optimization-Based Drift Counteraction Optimal Control: A Spacecraft Attitude Control Case Study}

\author{Sunbochen Tang \footnote{Graduate student, Department of Aerospace Engineering.}, Nan Li\footnote{Ph.D., Department of Aerospace Engineering.}, Robert A.E. Zidek\footnote{Ph.D., Department of Aerospace Engineering}, and Ilya Kolmanovsky\footnote{Professor, Department of Aerospace Engineering.}}
\affil{University of Michigan, Ann Arbor, Michigan 48109}

\begin{document}

\maketitle

\section*{Nomenclature}
{\renewcommand\arraystretch{1.0}
\noindent\begin{longtable*}{@{}l @{\quad=\quad} l@{}}
    $\mathcal{B}$ & body-fixed frame \\
    $C_{diff}$ & diffusion coefficient \\
    $\bar{g}$ & reaction wheel spin axis unit vector resolved in $\mathcal{B}$ frame \\
    $\mathcal{I}$ & inertial frame \\
    $J$ & inertia matrix of the cuboid bus in $\mathcal{B}$ frame \\
    $J_w$ & moment of inertia of reaction wheel about spin axis \\
    $l_i$ & distance between spacecraft's center of mass and geometric center; $i \in \{x, y, z\}$, m \\
    $L_i$ & dimension of cuboid bus; $i \in \{x, y, z\}$, m \\
    $p$ & number of operational reaction wheels \\
    $\Delta t$ & sampling period \\
    $u$ & control input vector \\
    $\hat{u}_S$ & unit vector in the direction from the spacecraft to the Sun \\
    $U$ & control admissible set \\
    $W$ & matrix containing reaction wheel spin axes as column vectors \\
    $x$ & state vector \\
    $X$ & prescribed set defined by state constraints \\
    $\kappa$ & time-before-exit \\
    $\lambda$ & Lagrange multiplier \\
    $\phi$ & Euler angle (roll), rad \\
    $\theta$ & Euler angle (pitch), rad \\
    $\psi$ & Euler angle (yaw), rad \\
    $\bar{\nu}$ & reaction wheel spin rate vector, rad/s\\
    $\bar{\omega}$ & angular velocity vector resolved in $\mathcal{B}$ frame, rad/s \\
    $\vec{\tau}_{srp}$ & solar radiation pressure torque vector, N$\cdot$m \\
    $\Phi_s$ & solar flux, W/m$^2$ \\
    \multicolumn{2}{@{}l}{Subscripts} \\
    E & Earth \\
    $k$ & time instant \\
    S & Sun \\
    SC & spacecraft \\
    srp	& solar radiation pressure
\end{longtable*}}

\section{Introduction}\label{sec:1}

Constraints are ubiquitous and crucial to the safe and efficient operation of many engineered systems. A variety of control strategies have been developed to manage constraints, such as model predictive control (MPC) \cite{camacho2013model}, reference governors \cite{garone2017reference}, and control barrier function \cite{ames2016control}. In certain scenarios, all trajectories of system state under admissible control inputs will eventually drift out of a desired operating region defined by a set of constraints, which makes constraint violation in finite time inevitable. It is not uncommon to observe such behaviors in systems subject to persistent disturbances and limited actuator capability/resources. For instance, a geostationary satellite will eventually fall out of its designated position because of orbital perturbations and limited amount of onboard fuel.

For aforementioned situations where constraint violation in finite time is inevitable, drift counteraction optimal control (DCOC) has been proposed in which one seeks to maximize a functional representing the total time or yield before the first occurrence of constraint violation \cite{kolmanovsky2018drift}. In this paper, we focus on a specific class of DCOC problems in which one seeks to maximize the time before system state exits a prescribed set defined by constraints. This time is referred to as the \textit{time-before-exit}. 

In the optimal control literature, DCOC problems are often referred to as ``exit-time'' problems or ``optimal stopping time'' problems \cite{barles1988exit}. For continuous-time deterministic systems, it has been shown that optimal control policy in such problems can be computed via finding the viscosity solutions of the corresponding Hamilton-Jacobi-Bellman (HJB) equation \cite{blanc1997deterministic}. However, it is generally difficult to solve an HJB equation either analytically or numerically, especially considering that it may not admit a smooth solution. As a result, recent advances in DCOC have been in a discrete-time setting \cite{kolmanovsky2008discrete,zidek2015deterministic,zidek2017drift,zidek2016stochastic,zidek2017auto,li2017training,zidek2017lp,zidek2017receding,zidek2018spacecraft}, which is more computationally tractable. Apart from computational benefits, a discrete-time formulation also leads to control solutions that are easily implementable with digital micro-controllers. For discrete-time systems, DCOC approaches based on dynamic programming (DP) \cite{kolmanovsky2008discrete,zidek2015deterministic,zidek2017drift,zidek2016stochastic,zidek2017auto,li2017training} and mixed-integer programming (MIP) \cite{zidek2017lp,zidek2017receding,zidek2018spacecraft} have been developed. However, both DP-based and MIP-based approaches are faced with computational challenges: The former can treat DCOC problems with general yield functional \cite{kolmanovsky2008discrete} but suffer from the \textit{curse of dimensionality} \cite{bertsekas1995dynamic}; this makes them computationally prohibitive to treat higher-order systems. An equivalent MIP reformulation of a discrete-time open-loop DCOC problem is proposed in \cite{zidek2017lp,zidek2017receding,zidek2018spacecraft}. However, because the number of integer variables in this MIP formulation is proportional to the planning horizon and the worst-case computational complexity of MIP grows combinatorially with the number of integer variables \cite{richards2005mixed}, computing numerical solutions to this MIP problem can be very challenging, especially for systems/problems that require longer planning horizons or shorter sampling periods. To alleviate the computational burden, an approach based on linear programming (LP) relaxation of the MIP formulation that replaces integer variables with continuous ones has been proposed in \cite{zidek2017lp} for linear systems. For nonlinear systems, an MPC-based strategy based on linearized prediction models and the LP reformulation is investigated in \cite{zidek2017receding,zidek2018spacecraft}. However, these continuous relaxation-based approaches only provide approximate solutions to the original DCOC problem and do not guarantee maximum \textit{time-before-exit}.

In this paper, we present a novel continuous optimization approach (i.e., based on optimization with only continuous variables) to discrete-time DCOC. In this approach, inspired by a recent result in minimum-time control \cite{verschueren2017stabilizing}, we employ exponentially weighted penalties in the cost function to encourage later constraint violation and show that this leads to guaranteed optimality of the control solution in terms of maximizing the \textit{time-before-exit}. This approach avoids the need for integer variables in optimization and is thus computationally more tractable for larger-dimensional problems. Compared to our preliminary conference paper \cite{tang2021continuous} published in AIAA sponsored American Control Conference, in this journal version we present a more complete treatment including re-worked and more streamlined theoretical results, the treatment of time-varying constraints, and a new case study that highlights the potential for application of these results to spacecraft attitude control.

High-precision attitude control is crucial for successfully performing imaging missions in deep space exploration. Because gas thrusters are often unable to achieve required precision and their use expends the onboard fuel, reaction wheels (RWs) are commonly used to control and/or maintain the spacecraft orientation in a prescribed region. However, RW failures are not uncommon, which can severely impair the spacecraft's mission effectiveness. For instance, two out of four RWs on the Kepler telescope malfunctioned \cite{cowen2013wheels}. When the number of functioning RWs is less than three, a control moment in an arbitrary direction cannot be generated and spacecraft becomes underactuated \cite{crouch1984spacecraft}. 
Under the assumption of zero angular momentum, a spacecraft with two independent RWs is shown to be small-time locally controllable in \cite{krishnan1995attitude}, and various strategies have been proposed for controlling such a spacecraft, including open-loop reorientaion methods \cite{rui2000nonlinear} and discontinuous feedback control \cite{horri2003attitude}. In the case of non-zero angular momentum, a procedure for constructing open-loop controls is introduced in \cite{boyer2007further}, and a method for recovering linear controllability by exploiting solar radiation pressure is proposed in \cite{petersen2016recovering}.

The DCOC can be used to maximize the time duration before spacecraft drifting out of a specified region desired for effectively accomplishing mission objectives (i.e., maximize the \textit{time-before-exit}). In this paper, we apply our continuous optimization approach to the original nonlinear spacecraft attitude dynamics model and demonstrate the effectiveness of our approach in maximizing the \textit{time-before-exit} for the cases of two and three functioning RWs in numerical examples. In \cite{zidek2017receding} and \cite{zidek2018spacecraft}, a DCOC approach based on LP relaxation of an MIP formulation is implemented based on a linearized spacecraft attitude dynamics model. Our work here improves upon the previous work of \cite{zidek2017receding} and \cite{zidek2018spacecraft} in the following two aspects: Firstly, when applied to a same model, our approach produces optimal solutions in terms of maximizing the \textit{time-before-exit}, while the approach of \cite{zidek2017receding} and \cite{zidek2018spacecraft} produces only approximate solutions to the original DCOC problem that may not achieve maximum \textit{time-before-exit}. Secondly, by applying our approach directly to the original nonlinear model, we avoid the errors due to linearization, which can negatively affect the performance of computed control solutions. Note that the LP-based approach of \cite{zidek2017receding} and \cite{zidek2018spacecraft} recomputes the control sequence over a receding horizon in an attempt to compensate for the mismatch between linear model and nonlinear system. Although our approach is directly applicable to more accurate nonlinear models and the principle of optimality \cite{bertsekas1995dynamic} holds for DCOC, a receding horizon approach could also be beneficial in our case if there is a significant model mismatch.

This paper is organized as follows: In Section~\ref{sec:2}, we introduce the general DCOC problem and present our continuous nonlinear programming (NLP)-based approach to DCOC. In Section~\ref{sec:3}, we establish theoretical results that verify the equivalence between the original DCOC problem and our NLP reformulation, and thereby verify the optimality of control solutions obtained by our approach in terms of maximizing the \textit{time-before-exit}. In Section~\ref{sec:4}, we consider the application of our DCOC approach to high-precision spacecraft attitude control, including a nominal case of a fully-actuated spacecraft with three RWs and a case of an underactuated spacecraft with two functioning RWs, where we also demonstrate the computational efficiency of our NLP-based approach. The paper is concluded in Section~\ref{sec:5}.

\section{A Continuous Optimization Approach to Drift Counteraction Optimal Control}\label{sec:2}

The objective of DCOC is to compute an optimal control that maximizes the time duration before the state vector exits a prescribed set. Consider a dynamic system that can be represented by the following discrete-time model,
\begin{equation}
    x_{k+1} = f_d(x_k, u_k),
\end{equation}
where $x_k \in \mathbb{R}^{n_x}$ denotes the state vector at the time instant $k$, and $u_k \in U \subset \mathbb{R}^{n_u}$ denotes the control input vector at $k$. We assume $f_d: \mathbb{R}^{n_x} \times \mathbb{R}^{n_u} \to \mathbb{R}^{n_x}$ is a twice continuously differentiable function (i.e., $f_d \in C^2(\mathbb{R}^{n_x} \times \mathbb{R}^{n_u} \to \mathbb{R}^{n_x})$).

Time-dependent state constraints are considered defined by the sets
\begin{equation}\label{equ:X_k}
    X(k) = \{x_k \in \mathbb{R}^{n_x}: H_k(x_k) \leq h_k\}, \quad k = 0, 1, 2, \dots, 
\end{equation}
where $H_k \in C^2(\mathbb{R}^{n_x} \to \mathbb{R}^{n_{h_k}})$ and $h_k \in \mathbb{R}^{n_{h_k}}$. When a finite horizon of length $N \in \mathbb{Z}$ is considered, we define the \textit{time-before-exit} as follows:
\begin{definition}\label{def:kappa}
    Given an initial condition $x_0 \in X(0)$ and control inputs $\{u_k\}_{k=0}^{N-1}$, the time-before-exit is defined as 
    \begin{equation}\label{eqn:kappa_1}
        \kappa (x_0, \{u_k\}_{k=0}^{N-1}) = \max \left\{k \in \mathbb{Z}_{[1,N]} : x_i \in X(i), i = 0, 1, \dots, k \right\}.
    \end{equation}
    Given an initial condition $x_0 \in X(0)$, the maximum time-before-exit is defined as 
    \begin{equation}\label{eqn:kappa_2}
        \kappa^*(x_0) = \max \left\{\kappa(x_0, \{u_k\}_{k=0}^{N-1}): u_k \in U, k = 0, 1, \dots, N-1 \right\}.
    \end{equation}
\end{definition}

\begin{remark}
    An alternative notion, ``first exit-time,'' is introduced in \cite{zidek2017receding}, which is defined as the first time instant where constraint violation occurs. As a result, the definition of ``first exit-time'' requires that constraint violation must happen during the considered finite horizon of length $N$ (i.e., $\kappa^*(x_0) + 1 \leq N$). In contrast, the notion ``time-before-exit'' considered in this paper enables us to treat a more general class of DCOC problems where constraints can be satisfied over the horizon (in this case $\kappa^*(x_0) = N$).
\end{remark}
 
Assuming the control admissible set $U \subset \mathbb{R}^{n_u}$ is compact and convex, the DCOC problem can be formally expressed as the following optimal control problem:
\begin{subequations}\label{equ:ocp}
    \begin{align}
        \max_{u_0, u_1, \dots, u_{N-1}} &\, \kappa (x_0, \{u_k\}_{k=0}^{N-1}) \\
        \text{subject to } &\, x_{k+1} = f_d(x_k, u_k), \\
        &\, u_k \in U, \quad k = 0, 1, \dots, N-1.
    \end{align}
\end{subequations}
As shown in \cite{zidek2017lp} (see also Remark~2), the above \eqref{equ:ocp} is a well-defined optimization problem.

We consider the following NLP problem with only continuous variables as a computationally efficient reformulation of the DCOC problem \eqref{equ:ocp}:
\begin{subequations}\label{equ:NLP}
    \begin{align}
        \min_{\substack{u_0, u_1, \dots, u_{N-1} \\ \epsilon_0, \epsilon_1, \dots, \epsilon_N}} &\, J_{\eqref{equ:NLP}} = \sum_{k = 0}^{N} \theta^{N-k} \epsilon_k \\
        \text{subject to } 
        &\, x_{k+1} = f_d(x_k, u_k), \label{equ:NLP_1} \\
        &\, u_k \in U, \\
        &\, 0 \leq \epsilon_k \leq \epsilon_{k+1}, \quad k = 0, 1, \dots, N-1, \label{equ:NLP_d} \\
        &\, H_k(x_k) \leq h_k + \textbf{1} M \epsilon_k, \quad k = 0, 1, \dots, N, \label{equ:NLP_2}
    \end{align}
\end{subequations}
where $M > 0$ is a sufficiently large positive number, and $\theta > 1$ is a weighting parameter and chosen to be sufficiently large. We will discuss theoretical properties of the above NLP problem \eqref{equ:NLP}, in particular, show the equivalence between \eqref{equ:ocp} and \eqref{equ:NLP} in the next section.


\begin{remark}
    An MIP-based reformulation of the DCOC problem \eqref{equ:ocp} has been proposed in \cite{zidek2017lp}, which uses binary variables in place of the continuous variables $\epsilon_k$ in \eqref{equ:NLP} and does not use exponential weighting in the cost function. Theorem~1 of \cite{zidek2017lp} establishes the equivalence between this MIP reformulation and the original DCOC problem \eqref{equ:ocp}. Compared to the MIP-based approach of \cite{zidek2017lp}, our continuous optimization approach (i.e., based on the continuous NLP problem \eqref{equ:NLP}) is significantly more computationally efficient while maintaining the equivalence and optimality guarantee in terms of maximizing the \textit{time-before-exit}.
\end{remark}

\section{Theoretical Results}\label{sec:3}

In this section, we discuss the relationship between the DCOC problem \eqref{equ:ocp} and our continuous NLP problem \eqref{equ:NLP}. The connection is established through two related optimization problems with only continuous variables. Based on sensitivity analysis and exact penalty method, we first derive intermediate results in Lemmas~\ref{lemma:dcoc-int}--\ref{lemma:absolute}, and then use them to prove our main results in Theorems~\ref{thm:main}--\ref{thm:convexity}. 

The first related problem is the following parameter-dependent optimization problem with the cost function $\phi(\epsilon_{\kappa^*(x_0)+1}, \dots, \epsilon_{N}) \triangleq \sum_{k = \kappa^*(x_0)+1}^{N}\theta^{N-k} \epsilon_k$:
\begin{subequations}\label{equ:eqcons-ocp}
\begin{align}
    \min_{\substack{u_0, u_1, \dots, u_{N-1} \\ \epsilon_0, \epsilon_1, \dots, \epsilon_N } } & \,\phi(\epsilon_{\kappa^*(x_0)+1}, \dots, \epsilon_{N}) \\
    \text{subject to }
    &\, x_{k+1} = f_d(x_k, u_k), \\
    &\, u_k \in U, \\
    &\, \epsilon_k \leq \epsilon_{k+1}, \quad k = \kappa^*(x_0), \dots, N-1, \\
    &\, H_k(x_k) \leq h_k + \textbf{1}M\epsilon_k, \quad k = 0, 1, \dots, N, \label{equ:eqcons-ocp_1} \\
    &\, \epsilon_{k} = \eta_k, \quad k = 0, 1, \dots, \kappa^*(x_0), \label{equ:eqcons-ocp_2}
\end{align}
\end{subequations}
where $\eta_k$, $k = 0, 1, \dots, \kappa^*(x_0)$, are parameters, and their nominal values are $\eta_k = 0$. To condense the notations, we define the following vectors: $\epsilon = [\epsilon_{\kappa^*(x_0)+1}, \dots, \epsilon_{N}]^\top$, $\eta = [\eta_0, \eta_1, \dots, \eta_{\kappa^*(x_0)}]^\top$, and $\eta^{(k)} = \eta_k e_k$, where $e_k$, $k = 0, 1, \dots, \kappa^*(x_0)$, are the standard basis vectors of $\mathbb{R}^{\kappa^*(x_0)+1}$. Note that when all parameters take their nominal values (i.e., $\eta_k = 0$ for all $k = 0, 1, \dots, \kappa^*(x_0)$), any feasible solution $z = (\{u_k\}_{k=0}^{N-1}, \{\epsilon_k\}_{k=0}^{N})$ must satisfy $\epsilon_k = 0$, for $k = 0, 1, \dots, \kappa^*(x_0)$, due to the constraints in \eqref{equ:eqcons-ocp_2}. Then, the following lemma connects \eqref{equ:eqcons-ocp} and the original DCOC problem \eqref{equ:ocp}:

\begin{lemma}\label{lemma:dcoc-int}
    For $\eta = \textbf{0}$, let $z = (\{u_k\}_{k=0}^{N-1}, \{\epsilon_k\}_{k=0}^{N})$ be a feasible solution to \eqref{equ:eqcons-ocp}. Then, $\{u_k\}_{k=0}^{N-1}$ is a global optimizer of the DCOC problem \eqref{equ:ocp}.
\end{lemma}
\begin{proof}
Because $\eta = \textbf{0}$ and $z$ is a feasible solution, \eqref{equ:eqcons-ocp_1} and \eqref{equ:eqcons-ocp_2} yield that $H_k(x_k) \leq h_k$ (i.e., $x_k \in X(k)$) holds for $k = 0, 1, \dots, \kappa^*(x_0)$. Then, according to \eqref{eqn:kappa_1} and \eqref{eqn:kappa_2} in Definition~\ref{def:kappa}, we have $\kappa (x_0, \{u_k\}_{k=0}^{N-1}) = \kappa^*(x_0)$, i.e., $\{u_k\}_{k=0}^{N-1}$ is a global optimizer of \eqref{equ:ocp}.
\end{proof}

Note that \eqref{equ:eqcons-ocp} cannot be constructed without the knowledge of the maximum \textit{time-before-exit} $\kappa^*(x_0)$, which is typically \textit{a-priori} unknown. Therefore, we now introduce the second related problem, which replaces the parameter-dependent equality constraints $\epsilon_k = \eta_k, k = 0, 1, \dots, \kappa^*(x_0)$, in \eqref{equ:eqcons-ocp_2} with penalty terms $\theta^{N-k}|\epsilon_k|$ in the cost function:
\begin{subequations}\label{equ:exact-ocp}
    \begin{align}
        \min_{\substack{u_0, u_1, \dots, u_{N-1} \\ \epsilon_0, \epsilon_1, \dots, \epsilon_{N}}} &\, J_{\eqref{equ:exact-ocp}} = \sum_{k = 0}^{\kappa^*(x_0)}\theta^{N-k} |\epsilon_k| + \phi(\epsilon_{\kappa^*(x_0)+1}, \dots, \epsilon_{N}) \label{equ:exact-ocp_1}\\
        \text{subject to }
        &\, x_{k+1} = f_d(x_k, u_k), \\
        &\, u_k \in U, \\
        &\, \epsilon_k \leq \epsilon_{k+1}, \quad k = \kappa^*(x_0), \dots, N-1, \label{equ:exact-ocp_d}\\
        &\, H_k(x_k) \leq h_k + \textbf{1}M\epsilon_k, \quad k = 0, 1, \dots, N.
    \end{align}
\end{subequations}
We now make the following two assumptions about \eqref{equ:eqcons-ocp}, which will facilitate establishing the relationship between \eqref{equ:eqcons-ocp} and \eqref{equ:exact-ocp}:

\begin{assumption}
    For a given $\eta \in \mathbb{R}^{\kappa^*(x_0)+1}$, we denote a minimizer of \eqref{equ:eqcons-ocp} by $z(\eta) = (\{u_k(\eta)\}_{k=0}^{N-1}, \{\epsilon_k(\eta)\}_{k=0}^N) \in \mathbb{R}^{n_z}$ and its associated Lagrange multiplier vector by $\lambda(\eta) \in \mathbb{R}^{n_{\lambda}}$. We assume that for $\eta = \textbf{0}$, the pair $\left(z(\textbf{0}), \lambda(\textbf{0})\right)$ satisfies the strong second-order sufficient conditions (see Theorem~2 of \cite{buskens2001sensitivity}).
    \label{assum:SSC}
\end{assumption}

Under Assumption~\ref{assum:SSC} and according to Theorem~3 of \cite{buskens2001sensitivity}, there exists a neighborhood of $\textbf{0}$, $V \subset \mathbb{R}^{\kappa^*(x_0)+1}$, and continuously differentiable $z(\eta): V \to \mathbb{R}^{n_z}$ and $\lambda(\eta): V \to \mathbb{R}^{n_{\lambda}}$ such that for all $\eta \in V$, the pair $\left(z(\eta), \lambda(\eta)\right)$ satisfies the strong second-order sufficient conditions. Furthermore, the following sensitivity result holds \cite{buskens2001sensitivity}: For $k = 0, 1, \dots, \kappa^*(x_0)$,
\begin{equation}\label{equ:sensitivity}
    \lim_{\eta_k \to 0^+} \frac{\phi(\epsilon(\eta^{(k)})) - \phi(\epsilon(\textbf{0}))}{\eta_k - 0} = -\lambda_k(\textbf{0}),
\end{equation}
where $\lambda_k(\textbf{0})$ is the Lagrange multiplier associated with the equality constraint $\epsilon_k = \eta_k = 0$.

\begin{assumption}\label{assum:sens}
    There exists $L > 0$ such that,
    \begin{equation}\label{equ:Lsens}
    \bigg\rvert \lim_{\eta_k \to 0^+} \frac{\epsilon_i(\eta^{(k)}) - \epsilon_i(\textbf{0})}{\eta_k - 0} \bigg\rvert \leq L,
    \end{equation}
    for all $i = \kappa^*(x_0)+1, \dots, N$, all $k = 0, 1, \dots, \kappa^*(x_0)$, and all $\theta > 1$.
\end{assumption}

Under Assumptions~\ref{assum:SSC} and \ref{assum:sens}, the relationship between \eqref{equ:eqcons-ocp} and \eqref{equ:exact-ocp} is stated in the following lemma:


\begin{lemma}\label{lemma:equal}
    Under Assumptions~\ref{assum:SSC} and \ref{assum:sens}, there exists $\theta_0 > 1$ such that if $\theta \geq \theta_0$, then \eqref{equ:eqcons-ocp} with $\eta = \textbf{0}$ and \eqref{equ:exact-ocp} share the same set of (local) minimizers.
\end{lemma}
\begin{proof}
See Lemma~1 of \cite{tang2021continuous}.
\end{proof}

Based on Lemma~\ref{lemma:equal}, we now state the relationship between \eqref{equ:NLP} and \eqref{equ:exact-ocp} in the following lemma:

\begin{lemma}\label{lemma:absolute}
    Under Assumptions~\ref{assum:SSC} and \ref{assum:sens}, when $\theta > 1$ is sufficiently large, \eqref{equ:NLP} and \eqref{equ:exact-ocp} share the same set of global minimizers.
\end{lemma}
\begin{proof}
    Let $z^* = (\{u^*_k\}_{k=0}^{N-1}, \{\epsilon^*_k\}_{k=0}^N)$ be a global minimizer of \eqref{equ:NLP}. Since the set of constraints of \eqref{equ:exact-ocp} is a subset of the constraints of \eqref{equ:NLP}, $z^*$ is a feasible point of \eqref{equ:exact-ocp}. Let $z^{\prime} = (\{u^{\prime}_k\}_{k=0}^{N-1}, \{\epsilon^{\prime}_k\}_{k=0}^N)$ be a global minimizer of \eqref{equ:exact-ocp}. By Lemma~\ref{lemma:equal}, $z^{\prime}$ is at least a local minimizer of \eqref{equ:eqcons-ocp} with $\eta = \textbf{0}$ and hence satisfies $\epsilon^{\prime}_{k} = \eta_k = 0$ for $k = 0, 1, \dots, \kappa^*(x_0)$. In turn, $z^{\prime}$ is a feasible point of \eqref{equ:NLP}. Since $z^*$ and $z^{\prime}$ are global minimizers of \eqref{equ:NLP} and \eqref{equ:exact-ocp}, respectively, and are feasible points of \eqref{equ:exact-ocp} and \eqref{equ:NLP}, respectively, we have $J_{\eqref{equ:NLP}}(z^*) \le J_{\eqref{equ:NLP}}(z^{\prime})$ and $J_{\eqref{equ:exact-ocp}}(z^{\prime}) \le J_{\eqref{equ:exact-ocp}}(z^*)$. However, for any point $z$ that is feasible to \eqref{equ:NLP} (and hence is also feasible to \eqref{equ:exact-ocp}), we have
    \begin{align*}
        J_{\eqref{equ:NLP}}(z) &= \sum_{k = 0}^{\kappa^*(x_0)}\theta^{N-k} \epsilon_k + \sum_{k = \kappa^*(x_0)+1}^{N}\theta^{N-k} \epsilon_k \\
        &= \sum_{k = 0}^{\kappa^*(x_0)}\theta^{N-k} |\epsilon_k| + \sum_{k = \kappa^*(x_0)+1}^{N}\theta^{N-k} \epsilon_k = J_{\eqref{equ:exact-ocp}}(z).
    \end{align*}
    Note that we have used $\epsilon_k \ge 0$ (due to the constraints in \eqref{equ:NLP_d}) to derive the equality. This implies $J_{\eqref{equ:NLP}}(z^*) \le J_{\eqref{equ:NLP}}(z^{\prime}) = J_{\eqref{equ:exact-ocp}}(z^{\prime}) \le J_{\eqref{equ:exact-ocp}}(z^*) = J_{\eqref{equ:NLP}}(z^*)$. Therefore, we have $J_{\eqref{equ:NLP}}(z^*) = J_{\eqref{equ:NLP}}(z^{\prime})$ and $J_{\eqref{equ:exact-ocp}}(z^{\prime}) = J_{\eqref{equ:exact-ocp}}(z^*)$, which implies that $z'$ and $z^*$ are not only feasible points but indeed also global minimizers of \eqref{equ:NLP} and \eqref{equ:exact-ocp}, respectively. This proves Lemma~\ref{lemma:absolute}.
\end{proof}

We are now ready to establish the connection between the DCOC problem \eqref{equ:ocp} and our NLP problem \eqref{equ:NLP} in the following theorem:

\begin{theorem}\label{thm:main}
    Given $x_0 \in X(0)$ and under Assumptions~\ref{assum:SSC} and \ref{assum:sens}, when $\theta > 1$ is sufficiently large, any global minimizer of \eqref{equ:NLP}, $(\{u^*_k\}_{k=0}^{N-1},\{\epsilon^*_k\}_{k=0}^N)$, must satisfy $\epsilon^*_k = 0$ for all $k = 0, 1,\dots,\kappa^*(x_0)$, where $\kappa^*(x_0)$ is the maximum \textit{time-before-exit} defined in \eqref{eqn:kappa_2}. In turn, the control input sequence $\{u^*_k\}_{k=0}^{N-1}$ is a global optimizer of the DCOC problem \eqref{equ:ocp}.
\end{theorem}
\begin{proof}
    Firstly, by Lemma~\ref{lemma:absolute}, any global minimizer of \eqref{equ:NLP}, $z^* = (\{u^*_k\}_{k=0}^{N-1}, \{\epsilon^*_k\}_{k=0}^N)$, must also be a global minimizer of \eqref{equ:exact-ocp}. Then, by Lemma~\ref{lemma:equal}, $z^*$ must be a (local) minimizer of \eqref{equ:eqcons-ocp} with $\eta = \textbf{0}$. In particular, $z^*$ is a feasible point of \eqref{equ:eqcons-ocp} and thus satisfies $\epsilon^*_k = \eta_k = 0$ for $k = 0, 1, \dots, \kappa^*(x_0)$. Finally, by Lemma~\ref{lemma:dcoc-int}, $\{u^*_k\}_{k=0}^{N-1}$ is indeed a global optimizer of the DCOC problem \eqref{equ:ocp}.
\end{proof}

Theorem~\ref{thm:main} says that a globally optimal solution to our NLP problem \eqref{equ:NLP} provides a control input sequence $\{u^*_k\}_{k=0}^{N-1}$ that
maximizes the \textit{time-before-exit} \eqref{eqn:kappa_1}, i.e., solves the DCOC problem \eqref{equ:ocp}. In practice, NLP problems are frequently solved using gradient-based algorithms (such as the interior-point method and the sequential quadratic programming method), which converge to only local minimizers. We now make the following additional assumption, which will enable us to extend Theorem~\ref{thm:main} to local minimizers:


\begin{assumption}\label{assum:convex}
    Let $Z \subset \mathbb{R}^{n_z}$ denote the feasible region of \eqref{equ:NLP} defined by the constraints \eqref{equ:NLP_1}-\eqref{equ:NLP_2}, and let $z^* = (\{u^*_k\}_{k=0}^{N-1}, \{\epsilon^*_k\}_{k=0}^N)$ be a global minimizer of \eqref{equ:NLP}. For any $z_0 \in Z$, there exists $r_0(z_0) > 0$ such that $z_0 + r (z^* - z_0) \in Z$ for all $r \in [0, r_0(z_0)]$.
\end{assumption}

Assumption~\ref{assum:convex} holds for many cases. For instance, if $f_d$ is linear and all components of $H_k$ are convex functions, then the feasible region $Z$ is convex, and in this case Assumption~\ref{assum:convex} holds true. More generally, if $Z$ is star-shaped with $z^*$ as a star center, then Assumption~\ref{assum:convex} also holds true. Under Assumption~\ref{assum:convex}, we have the following result:

\begin{theorem}\label{thm:convexity}
    Under Assumption~\ref{assum:convex}, local minimizers of \eqref{equ:NLP} are all global minimizers.
\end{theorem}
\begin{proof}
See Theorem~2 of \cite{tang2021continuous}.
\end{proof}

Theorem~\ref{thm:convexity} guarantees that, under Assumption~\ref{assum:convex}, any local minimizer of our NLP problem \eqref{equ:NLP} (e.g., obtained by a gradient-based algorithm) solves the DCOC problem \eqref{equ:ocp}.


\section{Case Study: Spacecraft Attitude Control}\label{sec:4}

High-precision pointing is desirable in many space missions including deep space telescope imaging and optical communication. High-precision pointing is often achieved using internal torque actuators such as reaction wheels (RWs). Failure of some RWs results in inaccessible spacecraft dynamics \cite{crouch1984spacecraft}, which may lead to failure of certain mission objectives. In this section, we present a case study of high-precision attitude control for a spacecraft with three RWs and an underactuated spacecraft with two functioning RWs. 

\subsection{Spacecraft attitude dynamics model}

We consider a spacecraft consisting of a cuboid bus and $p$ reaction wheels. An inertial frame $\mathcal{I}$ and a body-fixed frame $\mathcal{B}$ are used to represent spacecraft orientation. Throughout this section, a physical vector is denoted as $\vec{r}$, and its corresponding mathematical vector resolved in a frame $\mathcal{H}$ is denoted as $\vec{r}|_{\mathcal{H}}$. For simplicity, we denote the mathematical vector of $\vec{r}$ resolved in the body-fixed frame as $\Bar{r}$, i.e., $\Bar{r} = \vec{r}|_{\mathcal{B}}$. The spacecraft state and control input vectors are defined as
\begin{equation}
    x = \begin{bmatrix}
        \phi & \theta & \psi & \omega_1 & \omega_2 & \omega_3 & \nu_1 & \dots & \nu_p
    \end{bmatrix}^T, \quad u = \begin{bmatrix}
        \dot{\nu}_1 & \dots & \dot{\nu}_p
    \end{bmatrix}^T,
\end{equation}
where $\phi, \theta, \psi$ are the 3-2-1 Euler angles which characterize the sequence of rotations from the inertial frame $\mathcal{I}$ to the body-fixed frame $\mathcal{B}$, $\Bar{\omega} = [\omega_1, \omega_2, \omega_3]^T$ is the angular velocity vector of the spacecraft resolved in the body-fixed frame, $\Bar{\nu} = [\nu_1, \dots, \nu_p]^T$ are the spin rates of $p$ RWs, and $\dot{\nu}_1, \dots, \dot{\nu}_p$ are the angular acceleration of $p$ RWs. 

Assuming all $p$ RWs are identical and off-spin-axis moments of inertia of each RW are sufficiently small, let $J_w$ be the moment of inertia of each RW about its spin axis. In the body-fixed frame $\mathcal{B}$, the unit direction vector from the cuboid bus to the $i$th RW is denoted as $\Bar{g}_i$, and we define $W = [\Bar{g}_1, \dots, \Bar{g}_p]$. Let $J = diag(J_1, J_2, J_3)$ denote the inertia matrix of the cuboid bus in frame $\mathcal{B}$, and define the locked inertia as $\Bar{J} = J + J_w W W^T$. Based on Newton's second law for rotation and kinematic equations, the nonlinear dynamics model in continuous-time takes the form,
\begin{subequations}\label{equ:cont-dyn}
    \begin{align}
    \begin{bmatrix}
        \dot{\phi} \\ \dot{\theta} \\ \dot{\psi}
    \end{bmatrix} &= \frac{1}{\cos(\theta)}\begin{bmatrix}
        \cos(\theta) & \sin(\phi)\sin(\theta) & \cos(\phi)\sin(\theta)\\
        0 & \cos(\phi)\cos(\theta) & -\sin(\phi) \cos(\theta)\\
        0 & \sin(\phi) & \cos(\phi)
    \end{bmatrix} \begin{bmatrix}
        \omega_1 \\ \omega_2 \\ \omega_3
    \end{bmatrix}\\
    \dot{\Bar{\omega}} &= \Bar{J}^{-1}(\Bar{\tau}_{srp} - S[\Bar{\omega}](\Bar{J}\Bar{\omega} + J_w W \Bar{\nu}) - J_w W u)\\
    \dot{\Bar{\nu}} &= u
\end{align}
\end{subequations}
where $\Bar{\tau}_{srp}$ refers to the torque vector generated by solar radiation pressure (SRP) resolved in the body-fixed frame, which is a nonlinear function of the Euler angles $\phi, \theta, \psi$ and taken from \cite{petersen2016recovering}. Let $\dot{x} = f(x,u)$ be the short-hand notation for the continuous-time model \eqref{equ:cont-dyn}. Using Euler's method with sampling period $\Delta t$, a discrete-time model, $x_{k+1} = f_d(x_k, u_k) = x_k + f(x_k, u_k) \Delta t$, is obtained.

\subsection{Attitude control simulation results}

In the numerical experiments, we consider two scenarios where two or three RWs are operational, i.e. $p = 2, 3$. When $p = 3$, assuming the direction vectors of RW spinning axes are linearly independent, the reaction wheels are able to generate torque in an arbitrary direction in the body-fixed frame. When $p < 3$, the spacecraft becomes underactuated. In the sequel we consider the case $p = 2$ where only two RWs are operational to demonstrate the capability of our NLP-based approach to DCOC to maintain the attitude in the desired range for maximum time duration. The parameters used in the experiments are reported in Table~\ref{tab:parameter}.
\begin{table}[H]
\caption{Parameter values.}
    \centering
    \begin{tabular}{c c c}
    \hline
    \hline
       Quantity & Value & Unit \\
    \hline
        $J_1, J_2, J_3$  & 430, 1210, 1300 & $kg/m^2$\\
    \hline
        $J_w$ & 0.043 & $kg/m^2$\\
    \hline 
        $L_x, L_y, L_z$ & 2, 2.5, 5 & $m$\\
    \hline
        $l_x, l_y, l_z$ & 0, 0.5, 0 & $m$\\
    \hline
        $\Phi_s$ & 1367 & $W/m^2$\\
    \hline
        $C_{diff}$ & 0.2 & [-]\\
    \hline 
        $\hat{u}_{S}$ & $[1/\sqrt{3}, 1/\sqrt{3}, 1/\sqrt{3}]^T$ & [-]\\
    \hline
    \hline
    \end{tabular}
    \label{tab:parameter}
\end{table}

\subsubsection{Three reaction wheels: $p = 3$}

In the numerical experiments with three operational reaction wheels, the unit vectors of the RW spinning axis directions resolved in the body-fixed frame $\mathcal{B}$ are defined as
\begin{equation*}
    W = [\Bar{g}_1, \Bar{g}_2, \bar{g}_3] = \begin{bmatrix}
        1 & 0 & 0 \\ 0 & 1 & 0 \\ 0 & 0 & 1
    \end{bmatrix}^T.
\end{equation*}

The following constraints for state and control input variables are considered,
\begin{align*}
    -0.003 [rad] &\leq \phi \leq 0.002 [rad], \quad -0.00065 [rad] \leq \theta \leq 0.00135 [rad], \quad -0.01 [rad] \leq \psi \leq 0.01 [rad]\\
    20 [rad/s]&\leq ||\Bar{\nu}||_1 \leq 80 [rad/s], \quad 0 [rad/s^2]\leq ||u||_1 \leq 2 [rad/s^2]
\end{align*}
where $||\cdot||_1$ is the 1-norm of a vector. Note that the spinning rates of three RWs are all lower bounded by $20\ [rad/s]$, which prevents zero speed crossings and additional RW wear. The initial condition for state variables are chosen as 
\begin{align*}
    [\phi_0, \theta_0, \psi_0] &= [-1\times 10^{-3}, 3.5\times 10^{-4}, -5\times 10^{-4}] \quad [rad]\\
    [\Bar{\omega}_{1, 0}, \Bar{\omega}_{2, 0}, \Bar{\omega}_{3, 0}] &= [-5, 2, 5]\times 10^{-4} \quad [rad/s]\\
    [\Bar{\nu}_{1, 0}, \Bar{\nu}_{2, 0}, \bar{\nu}_{3, 0}] &= [50, 50, 50] \quad [rad/s]
\end{align*}

We choose a sampling period of $\Delta t = 2\ [sec]$, and a prediction horizon of $N = 75$. Using ``ipopt'' solver and $\theta = 1.1$, the numerical simulation results of the proposed NLP approach to DCOC \eqref{equ:NLP} are reported in Fig.~\ref{fig:3RW_1} below.
\begin{figure}[H]
\begin{center}
\begin{picture}(400, 370)
\put(  0,  185){\epsfig{file=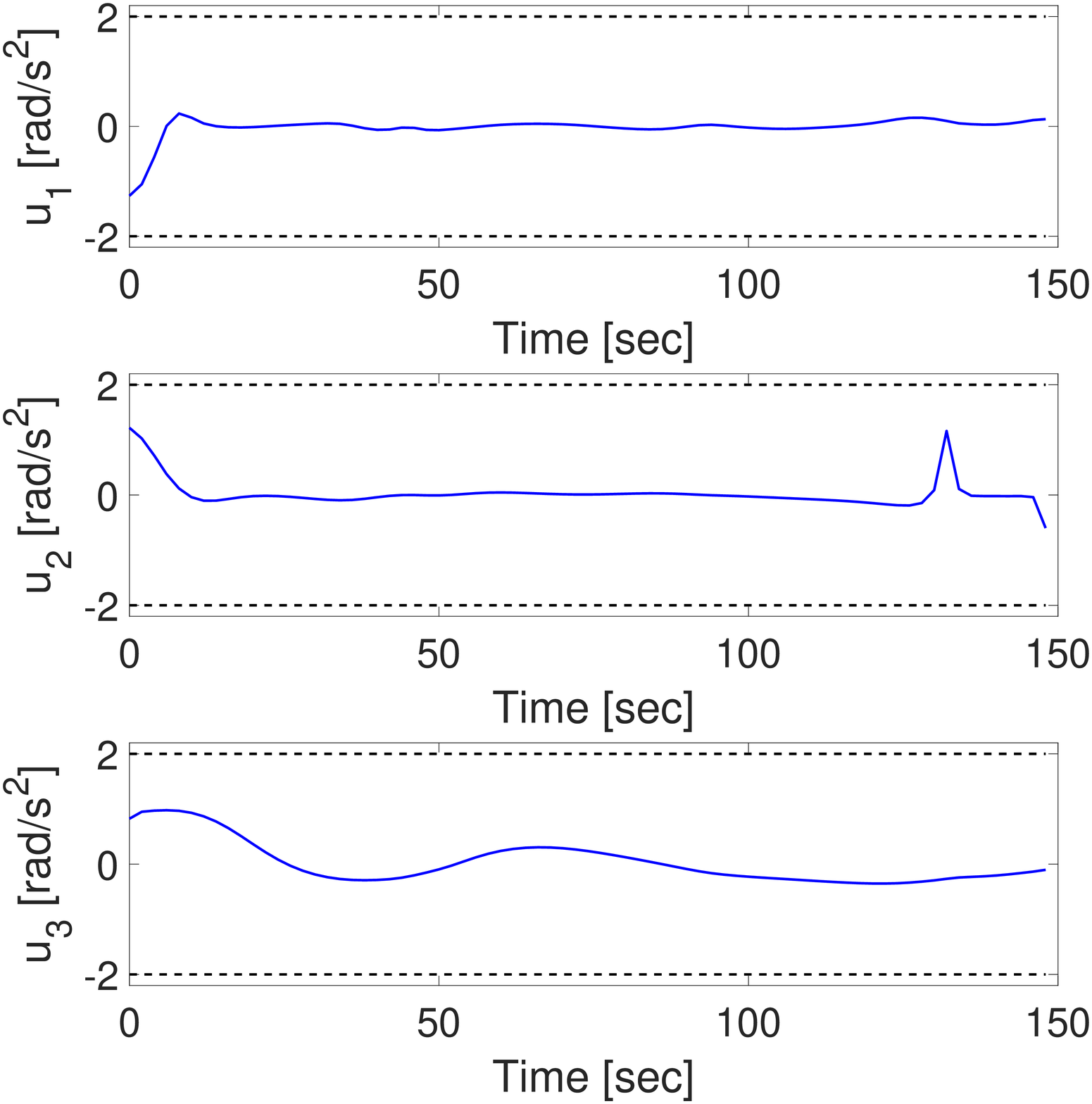,height=2.5in}}  
\put(  200,  185){\epsfig{file=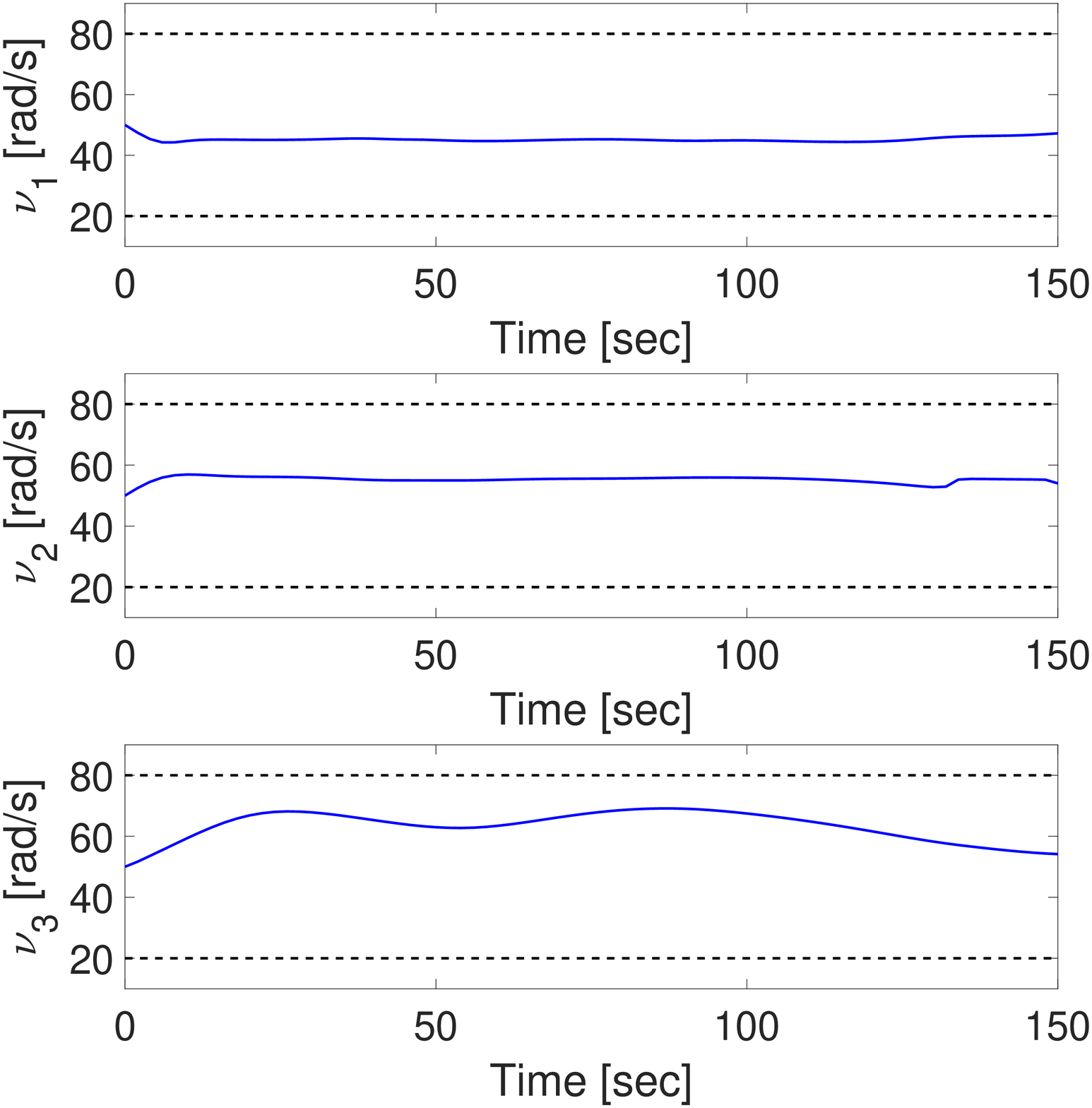,height=2.5in}}  
\put(  0,  0){\epsfig{file=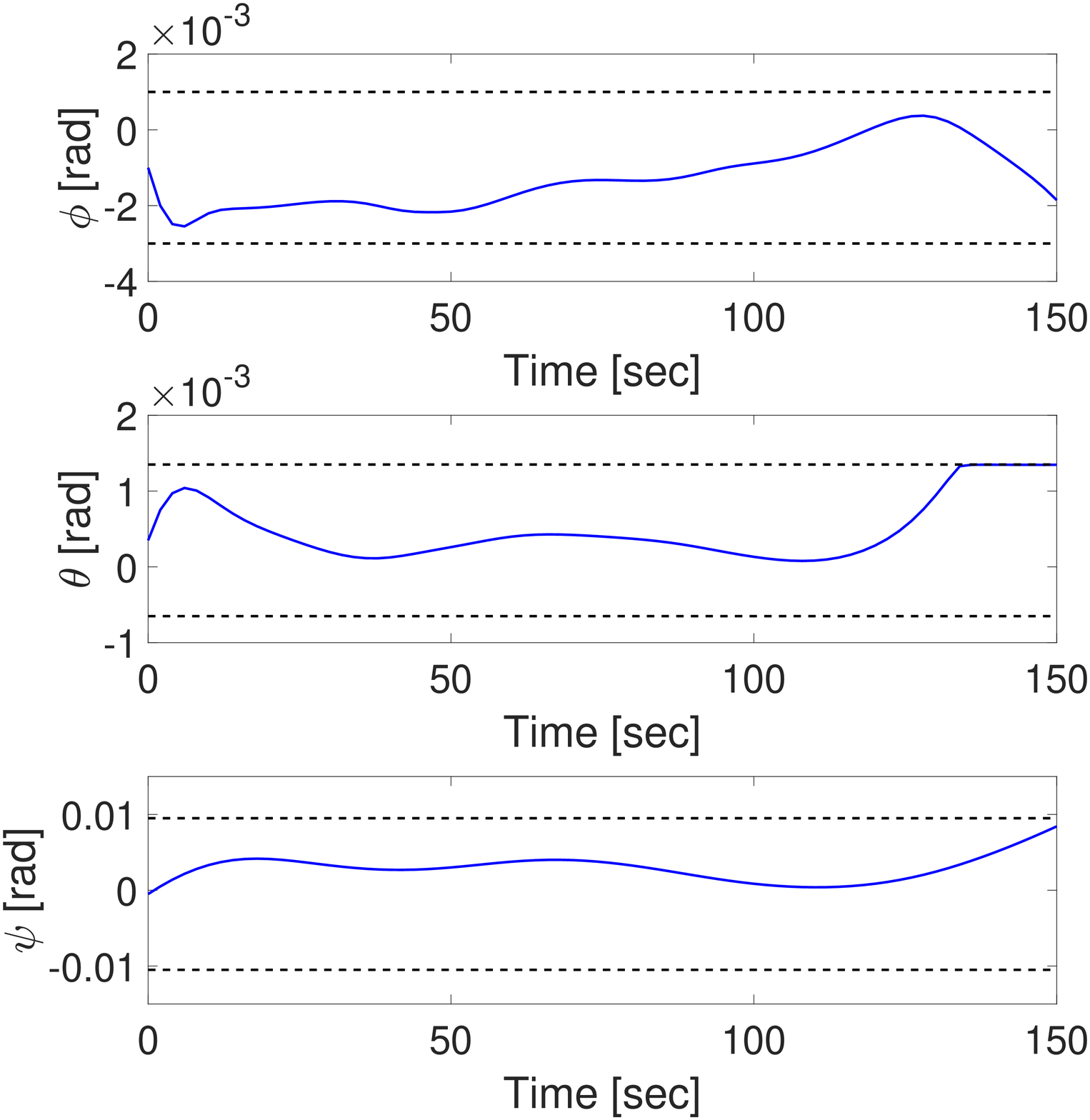,height=2.5in}}  
\put(  200, 0){\epsfig{file=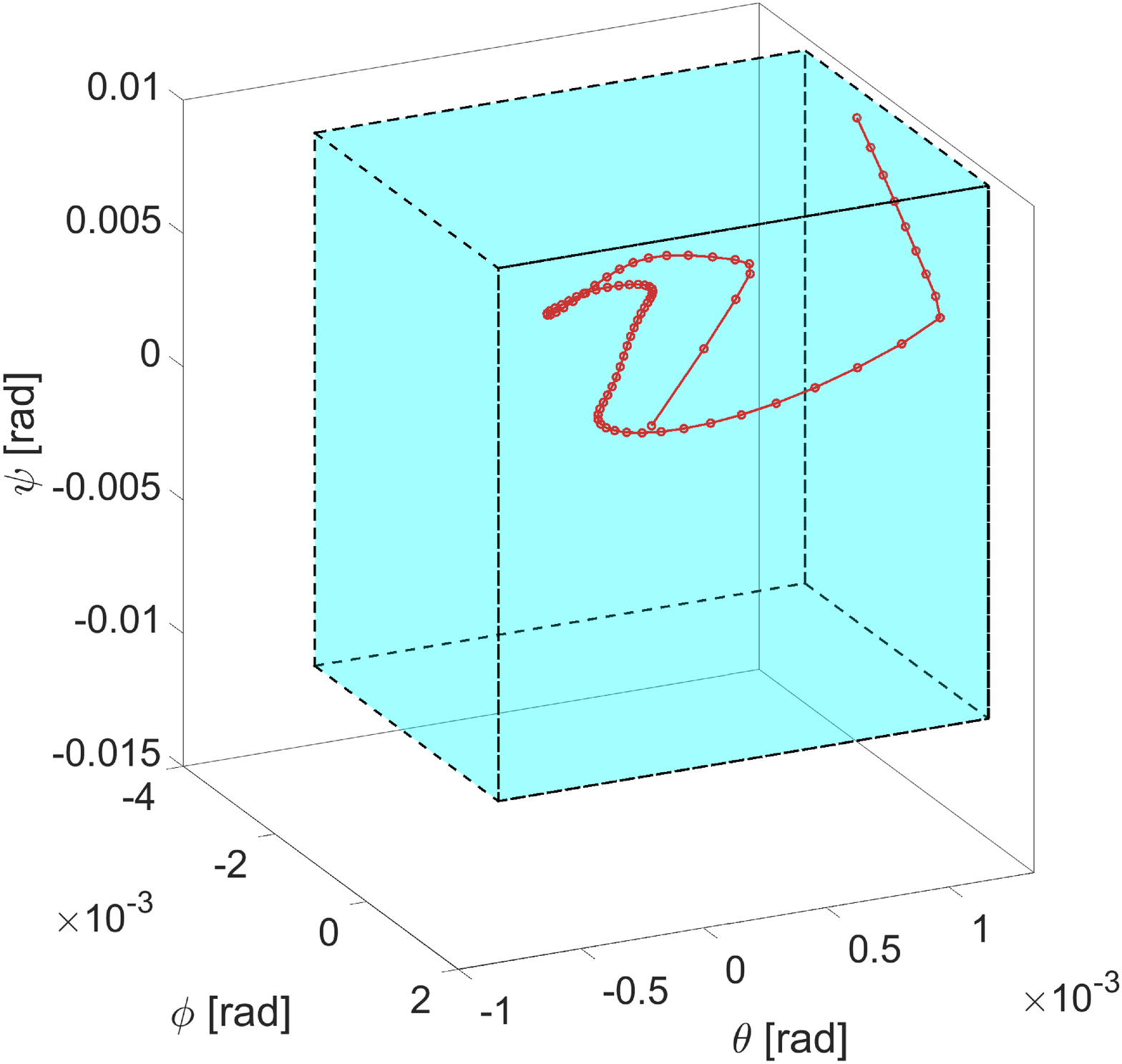,height=2.5in}}
\small
\put(180, 360){(a)}
\put(380, 360){(b)}
\put(180, 180){(c)}
\put(380, 180){(d)}
\normalsize
\end{picture}
\end{center}
      \caption{High precision pointing with three operational RWs (No constraint violation).}
      \label{fig:3RW_1}
      \vspace{-0.2in}
\end{figure}

The time histories of control variables, spinning accelerations of each reaction wheel are plotted in (a). The time histories of the reaction wheel spinning rates, and the Euler angles are presented in (b) and (c), respectively. A trajectory of Euler angles ($\phi, \theta, \psi$) is plotted as the red line in (d), where the blue cubic region represents the desired operating ranges for Euler angles. During the prediction horizon of $150\ [sec]$, our optimal control policy manages to keep all state variables within given ranges. Such results are expected since three RWs are sufficient to generate torque in an arbitrary direction to counteract the non-zero angular momentum of the spacecraft. 

In the next example, we present a simulation with same condition as the previous one except the initial spinning rate of the second reaction wheel $\bar{\nu}_{2, 0}$. The initial spinning rate vector of three reaction wheels are given as,
\begin{equation*}
    [\Bar{\nu}_{1, 0}, \Bar{\nu}_{2, 0}, \bar{\nu}_{3, 0}] = [50, 75.2, 50] \quad [rad/s]
\end{equation*}

The simulation results with modified initial condition are shown in Fig.~\ref{fig:3RW_2} below.
\begin{figure}[H]
\begin{center}
\begin{picture}(400, 370)
\put(  0,  185){\epsfig{file=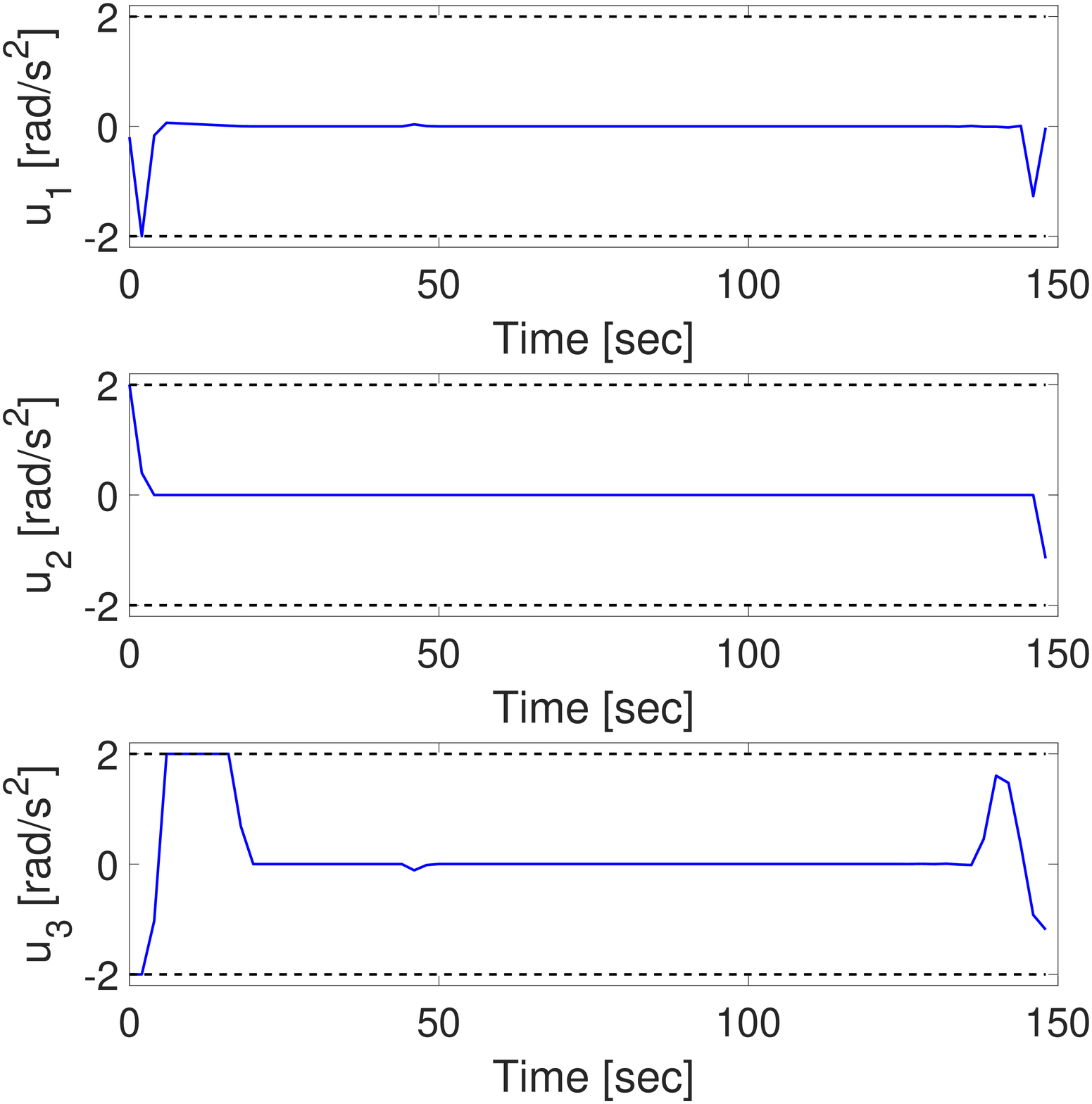,height=2.5in}}  
\put(  200,  185){\epsfig{file=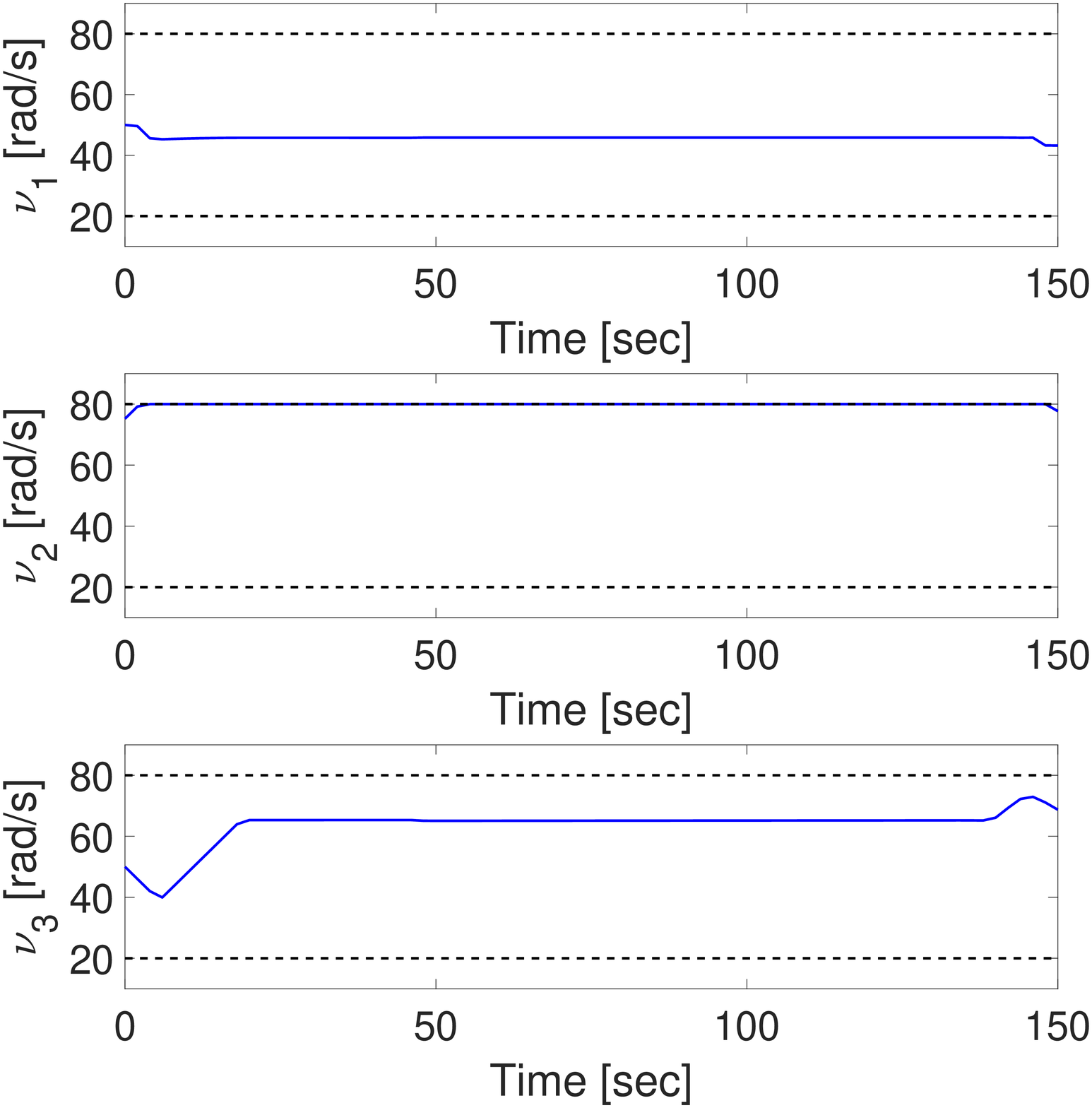,height=2.5in}}  
\put(  0,  0){\epsfig{file=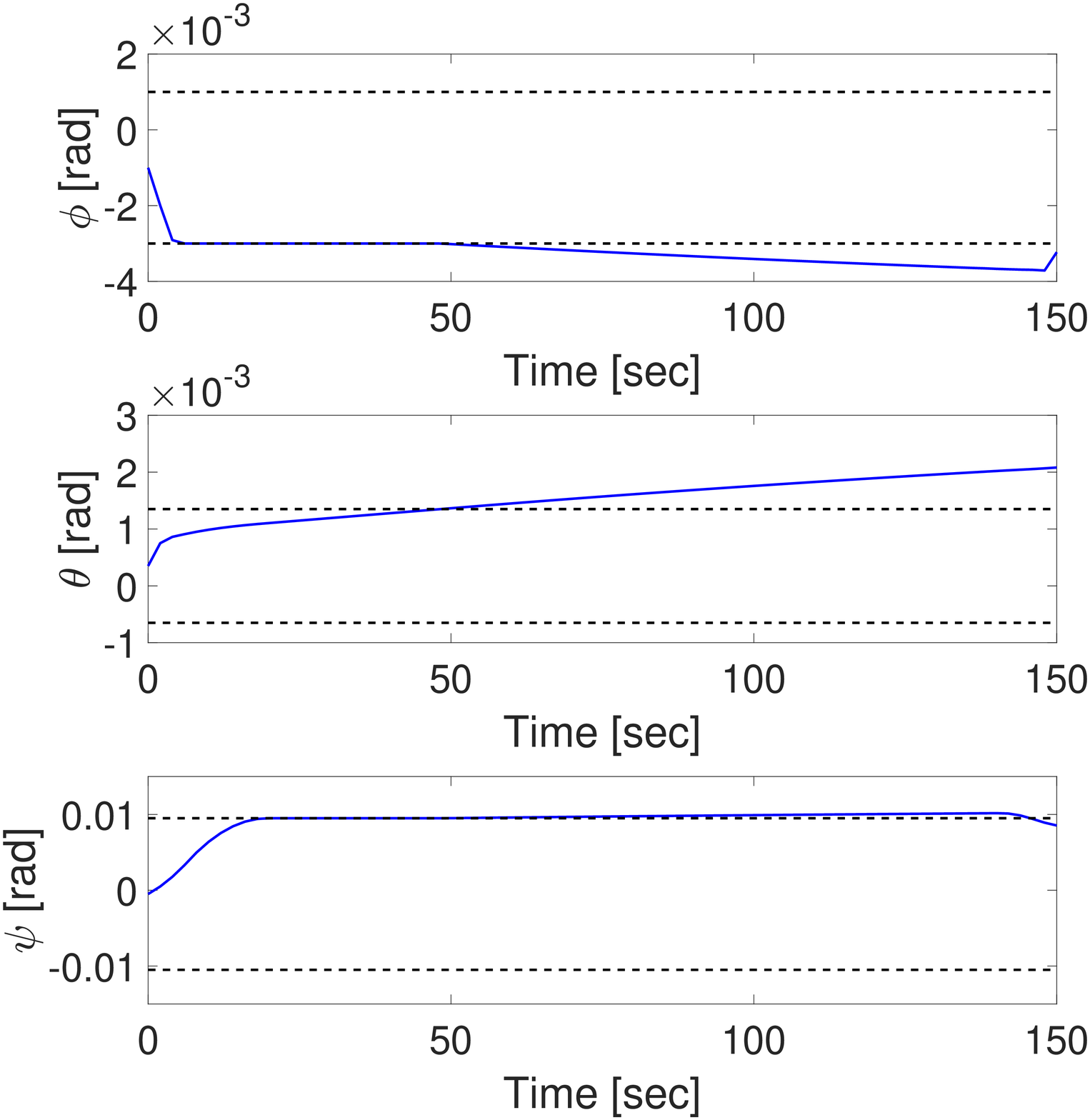,height=2.5in}}  
\put(  200, 0){\epsfig{file=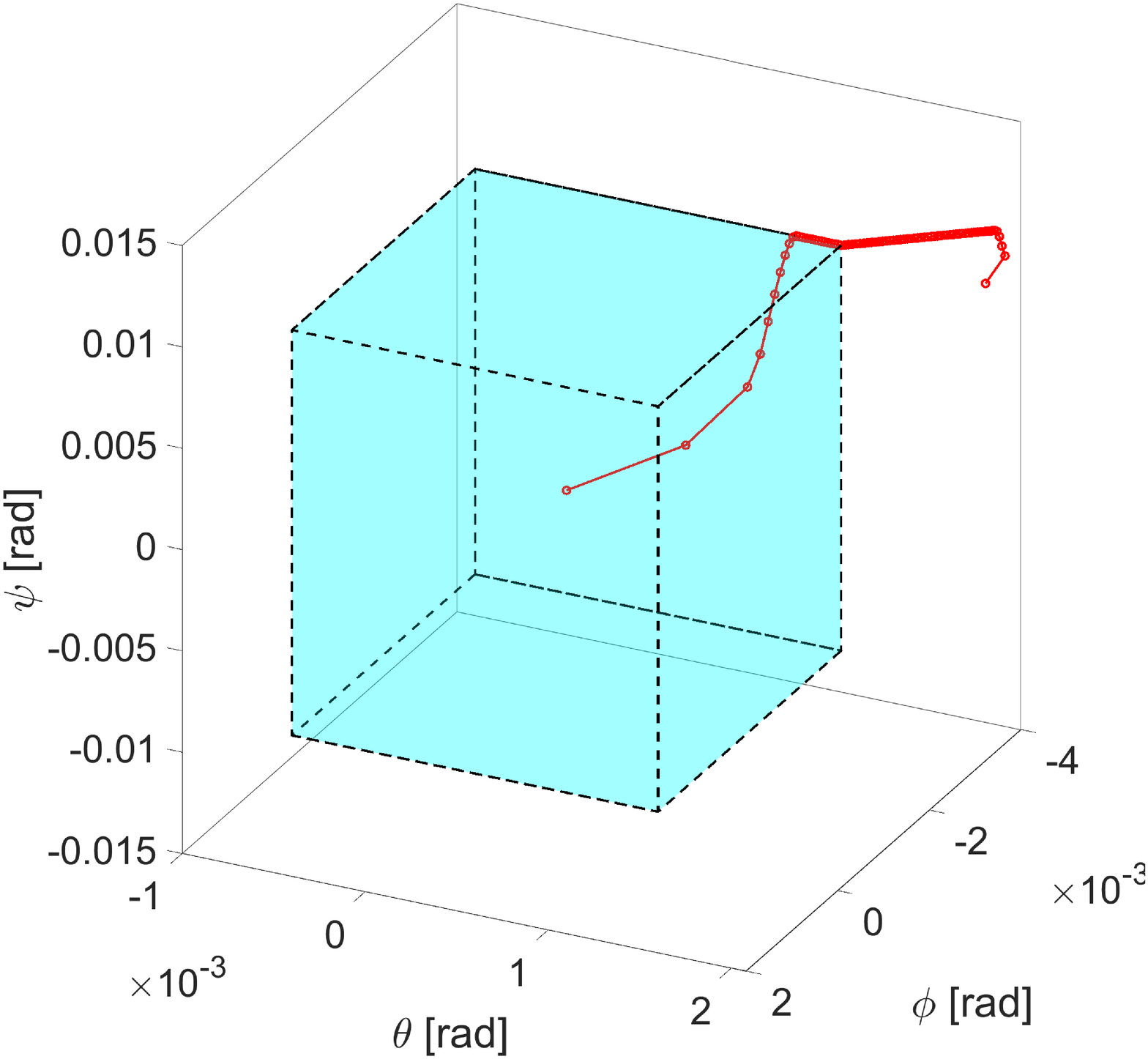,height=2.5in}}
\small
\put(180, 360){(a)}
\put(380, 360){(b)}
\put(180, 180){(c)}
\put(380, 180){(d)}
\normalsize
\end{picture}
\end{center}
      \caption{High precision pointing with three operational RWs (Constraint violation observed).}
      \label{fig:3RW_2}
      \vspace{-0.2in}
\end{figure}

The modified second RW initial spinning rate $\bar{\nu}_{2, 0}$ is close to its operating limit, 80 [$rad/s$]. With the modification, Euler angles cannot be kept within the specified range over the prediction horizon. Specifically, $\phi$ and $\theta$ first violate their constraints at $t = 48\ [sec]$, and $\psi$ exits its desired range shortly afterwards at $t = 52\ [sec]$ as shown in (c). It is observed in (b) that the second reaction wheel reaches its maximum spinning rate at the beginning of the simulation, which restricts the RWs to generate sufficient torque to counteract the angular momentum. Despite the inevitable constraint violation, the optimal control slows down the Euler angles from drifting out of the prescribed sets as shown in (c), for example, $\phi$ is maintained at its lower bound without exiting the desired region between $8\ [sec]$ and $48\ [sec]$.

\subsubsection{Two reaction wheels: $p = 2$}

Now we consider an underactuated spacecraft which only has two operational reaction wheels. We present two cases with different initial conditions, control admissible sets, reaction wheel spinning axis directions, and state constraints. In the first case, the unit vectors of RW spinning axis directions in the body-fixed frame $\mathcal{B}$ are 
\begin{equation*}
    W = [\Bar{g}_1, \Bar{g}_2] = \begin{bmatrix}
        1/\sqrt{3} & 1 \\ 1/\sqrt{3} & 0 \\ 1/\sqrt{3} & 0
    \end{bmatrix}^\top.
\end{equation*}

The state and control input constraints are defined as,
\begin{align*}
    -0.003 [rad] &\leq \phi \leq 0.002 [rad], \quad -0.0014 [rad] \leq \theta \leq 0.0026 [rad], \quad -0.02 [rad] \leq \psi \leq 0.02 [rad]\\
    20 [rad/s]&\leq ||\Bar{\nu}||_1 \leq 80 [rad/s], \quad 0 [rad/s^2]\leq ||u||_1 \leq 4 [rad/s^2]
\end{align*}

The initial condition for state variables are chosen as,
\begin{align*}
    [\phi_0, \theta_0, \psi_0] &= [-1 \times 10^{-3}, 6\times 10^{-4}, -5\times 10^{-4}] \quad [rad]\\
    [\Bar{\omega}_{1, 0}, \Bar{\omega}_{2, 0}, \Bar{\omega}_{3, 0}] &= [-5, 2, 3]\times 10^{-4} \quad [rad/s]\\
    [\Bar{\nu}_{1, 0}, \Bar{\nu}_{2, 0}] &= [50, 50] \quad [rad/s]
\end{align*}

The sampling period is $\Delta t = 2\ [sec]$, and the prediction horizon length is $N = 75$. Using `ipopt' solver and parameter $\theta = 1.1$, the numerical simulation results of our NLP approach to DCOC are plotted in Fig.~\ref{fig:2RW_1} below.
\begin{figure}[H]
\begin{center}
\begin{picture}(400, 360)
\put(  0,  180){\epsfig{file=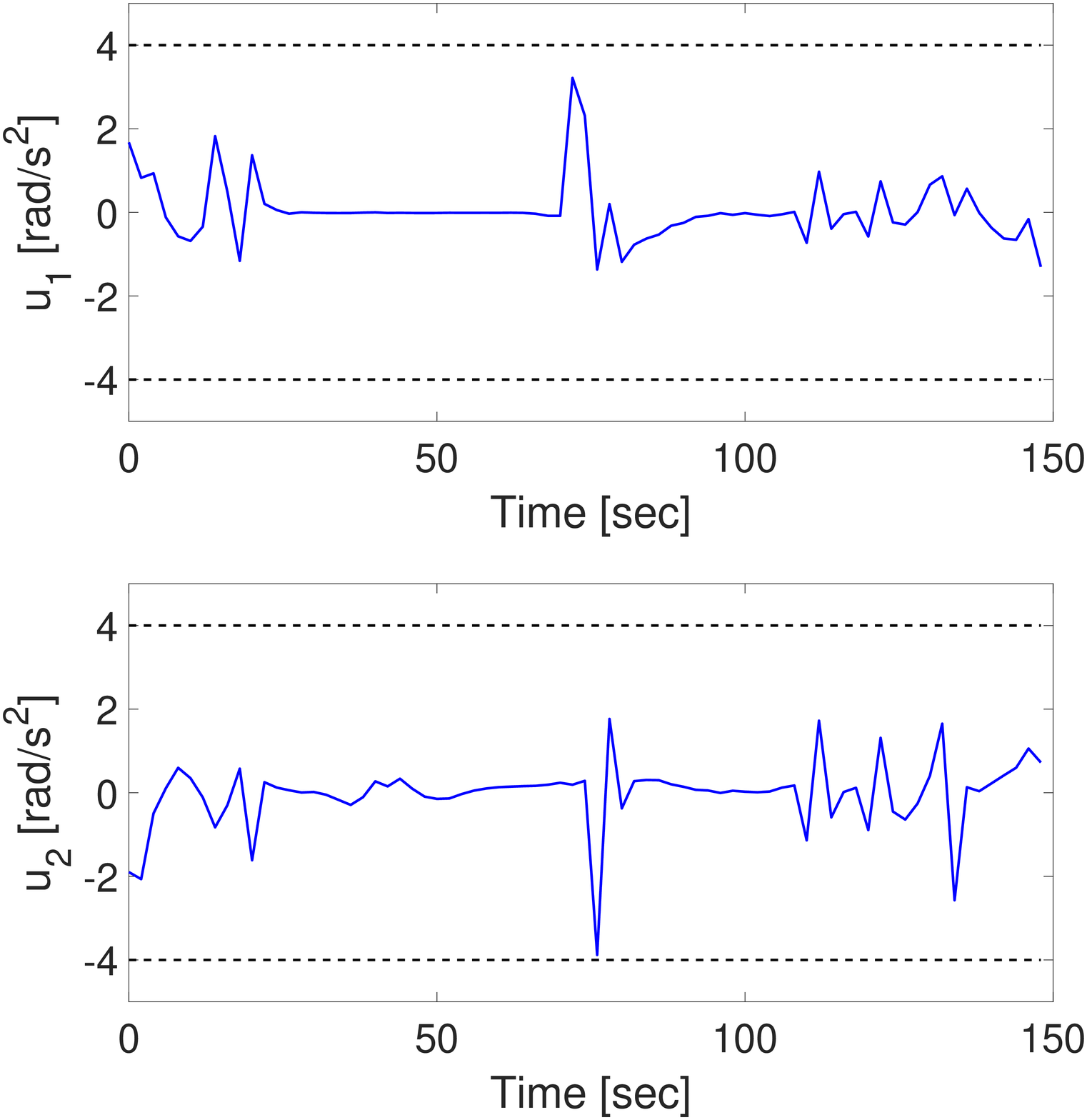,height=2.4in}}  
\put(  200,  180){\epsfig{file=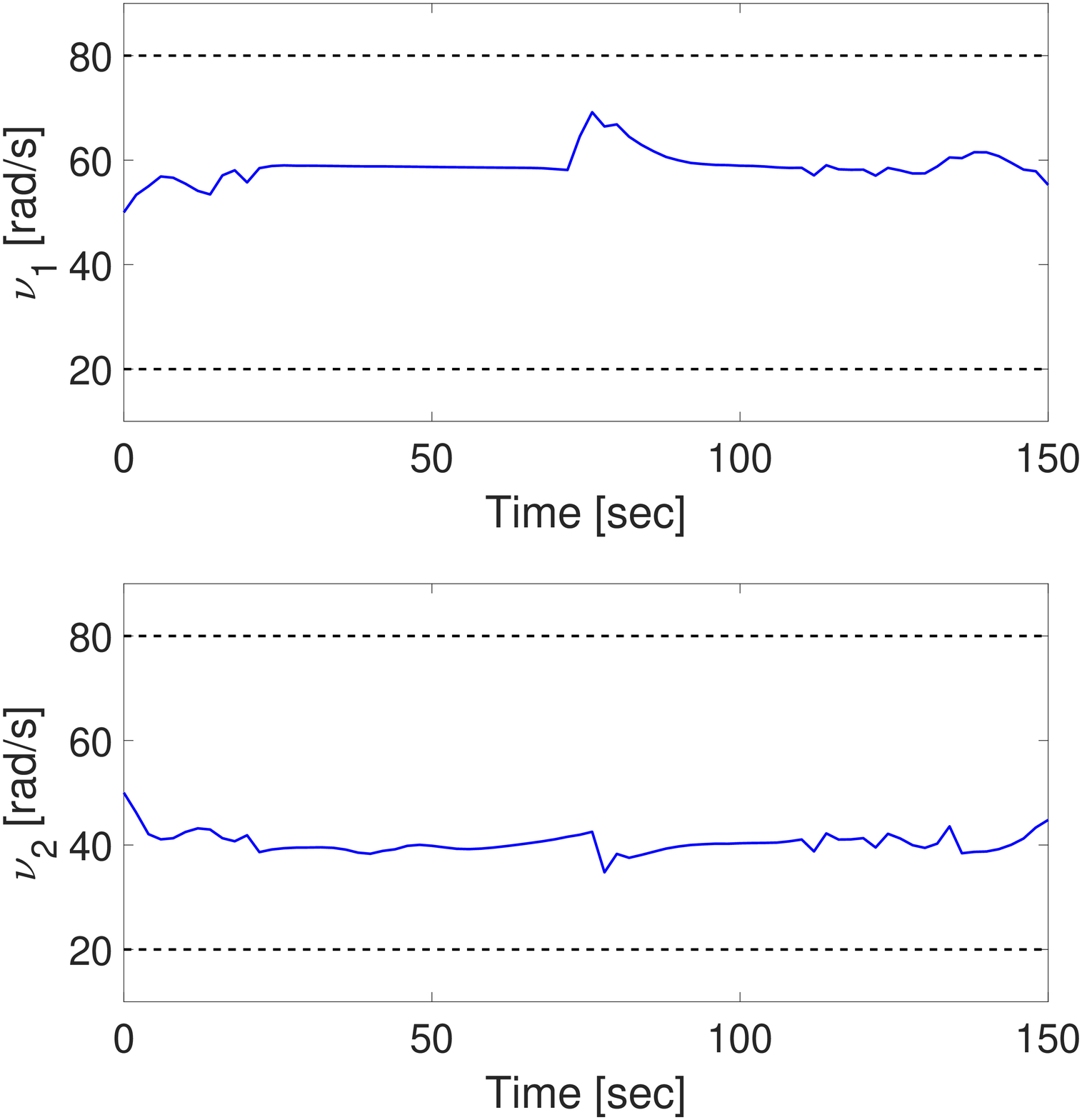,height=2.4in}}  
\put(  0,  0){\epsfig{file=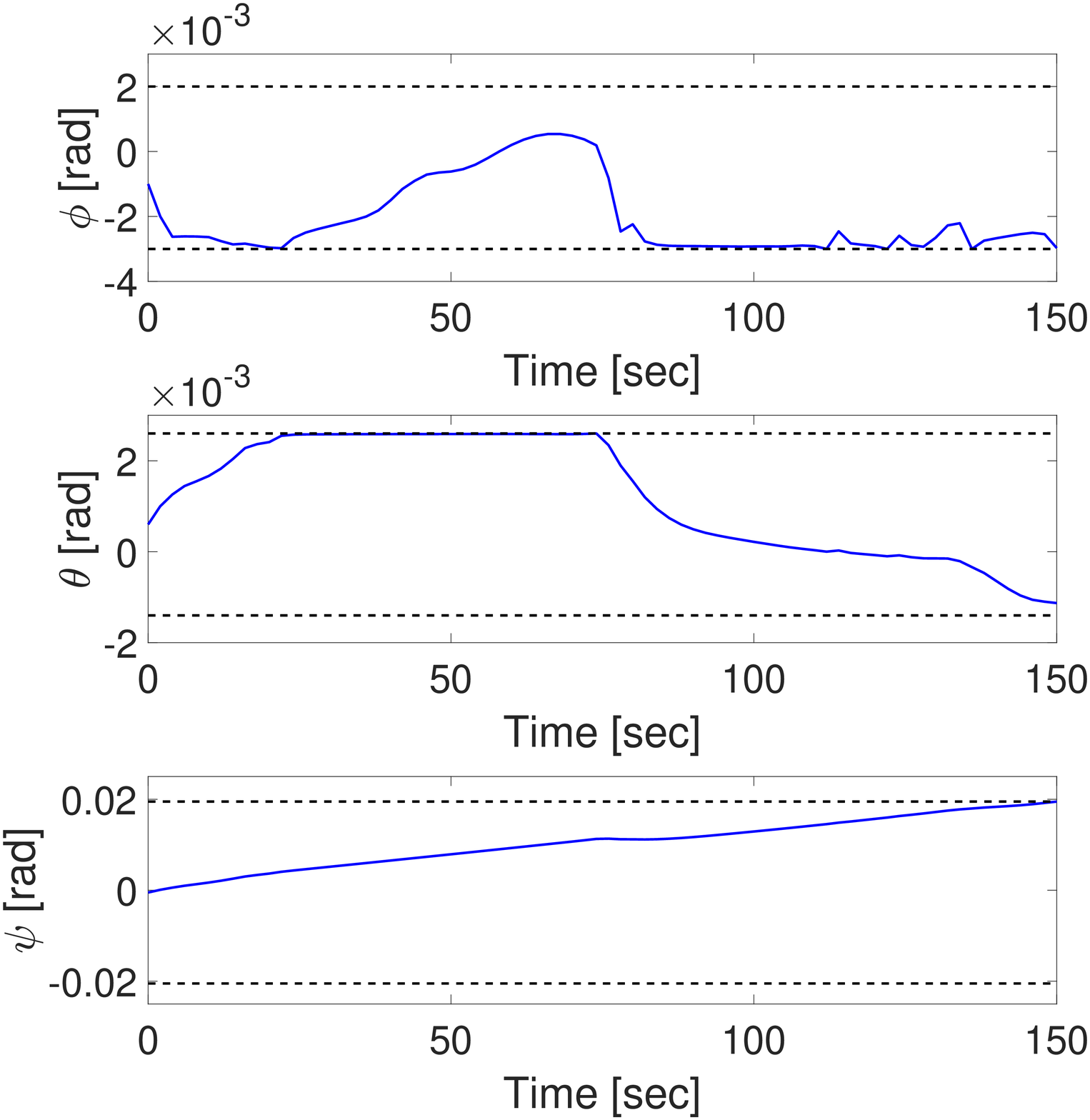,height=2.4in}}  
\put(  200, 0){\epsfig{file=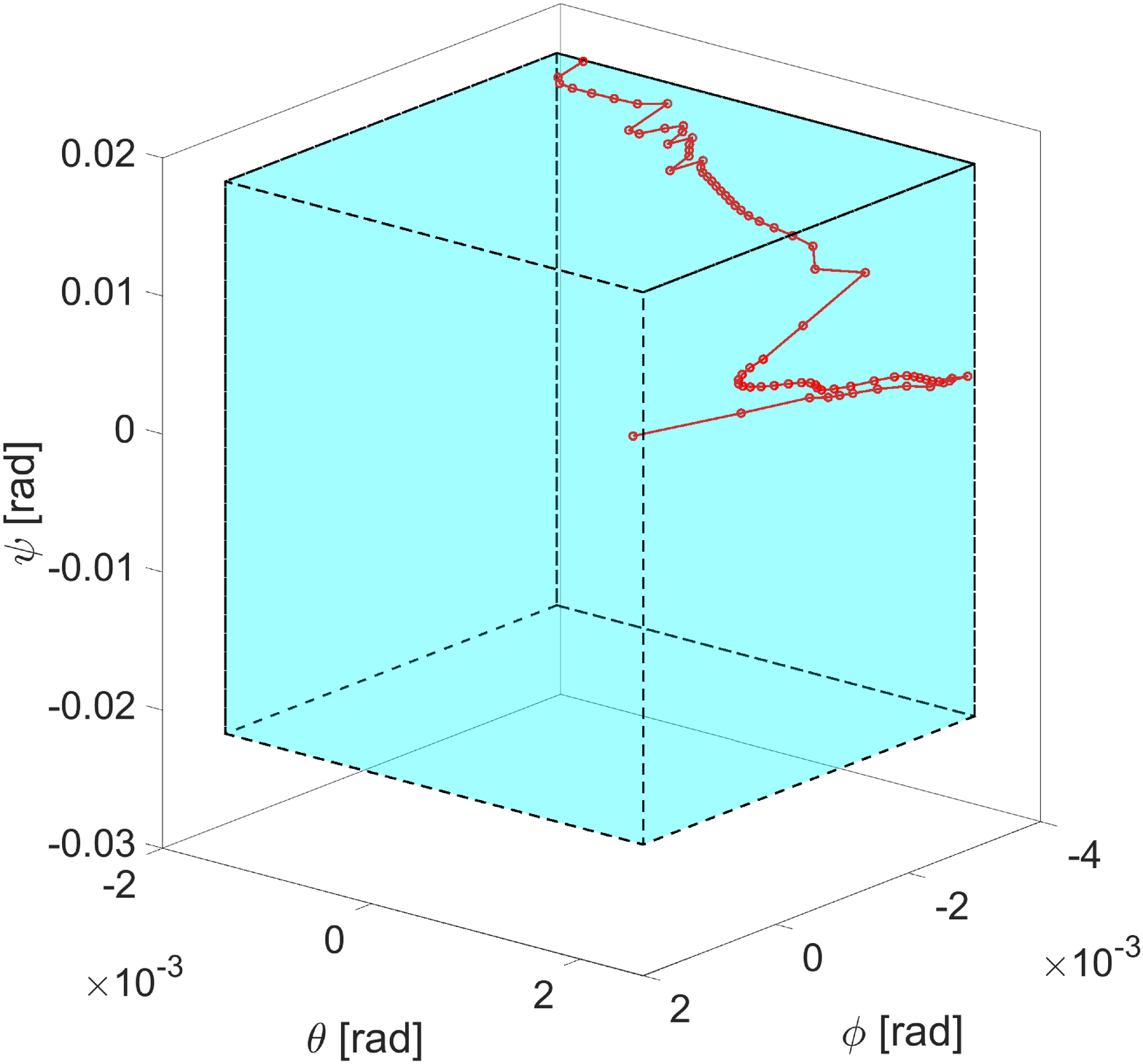,height=2.4in}}
\small
\put(180, 340){(a)}
\put(380, 340){(b)}
\put(180, 165){(c)}
\put(380, 165){(d)}
\normalsize
\end{picture}
\end{center}
      \caption{High precision pointing with two operational RWs (No constraint violation).}
      \label{fig:2RW_1}
      \vspace{-0.1in}
\end{figure}

Despite the spacecraft is underactuated, in this example, DCOC still manages to maintain all state variables within their desired range over the prediction horizon. Although two reaction wheels can only generate torque on a 2-D plane in the body-fixed frame, it appears that there is some coordination between the two reaction wheels that allows them to counteract angular momentum in all three axes. Specifically, spikes are observed in both RW accelerations (as shown in (a)) around $t = 70\ [sec]$. At the same time instant, it is observed in (c) that $\psi$ is maintained at current value for about $15\ [sec]$ before it keeps growing, which avoids constraint violation over the prediction horizon.

In the second case, the unit vectors of the RW spinning axis directions resolved in the body-fixed frame $\mathcal{B}$ are
\begin{equation*}
    W = [\Bar{g}_1, \Bar{g}_2] = \begin{bmatrix}
        1/\sqrt{3} & 0 \\ 1/\sqrt{3} & 1 \\ 1/\sqrt{3} & 0
    \end{bmatrix}^\top.
\end{equation*}

The following constraints for state and control input variables are considered,
\begin{align*}
    0.09 [rad] &\leq \phi \leq 1.01 [rad], \quad -0.02 [rad] \leq \theta \leq 0.02 [rad], \quad -0.05 [rad] \leq \psi \leq 0.05 [rad]\\
    20 [rad/s]&\leq ||\Bar{\nu}||_1 \leq 100 [rad/s], \quad 0 [rad/s^2]\leq ||u||_1 \leq 1 [rad/s^2]
\end{align*}

The initial condition for state variables are chosen as 
\begin{align*}
    [\phi_0, \theta_0, \psi_0] &= [1, 3\times 10^{-4}, -0.01] \quad [rad]\\
    [\Bar{\omega}_{1, 0}, \Bar{\omega}_{2, 0}, \Bar{\omega}_{3, 0}] &= [4, 4, -50]\times 10^{-5} \quad [rad/s]\\
    [\Bar{\nu}_{1, 0}, \Bar{\nu}_{2, 0}] &= [80, 20] \quad [rad/s]
\end{align*}

The sampling period is $\Delta t = 2\ [sec]$, and the prediction horizon length is $N = 75$. Using `ipopt' solver and parameter $\theta = 1.1$, the numerical simulation results of our NLP approach to DCOC are plotted in Fig.~\ref{fig:2RW_2} below.
\begin{figure}[H]
\begin{center}
\begin{picture}(400, 370)
\put(  0,  185){\epsfig{file=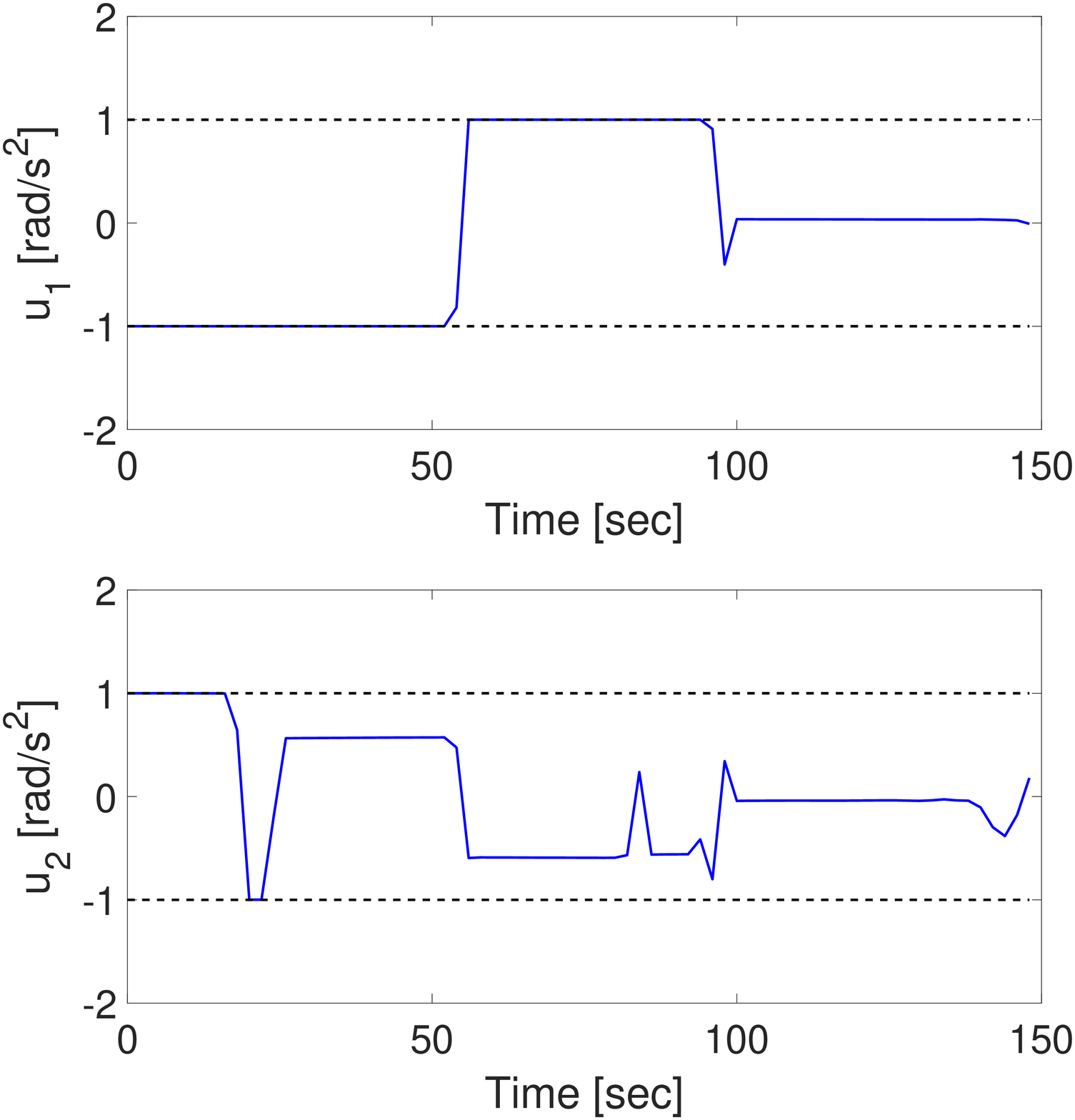,height=2.4in}}  
\put(  200,  185){\epsfig{file=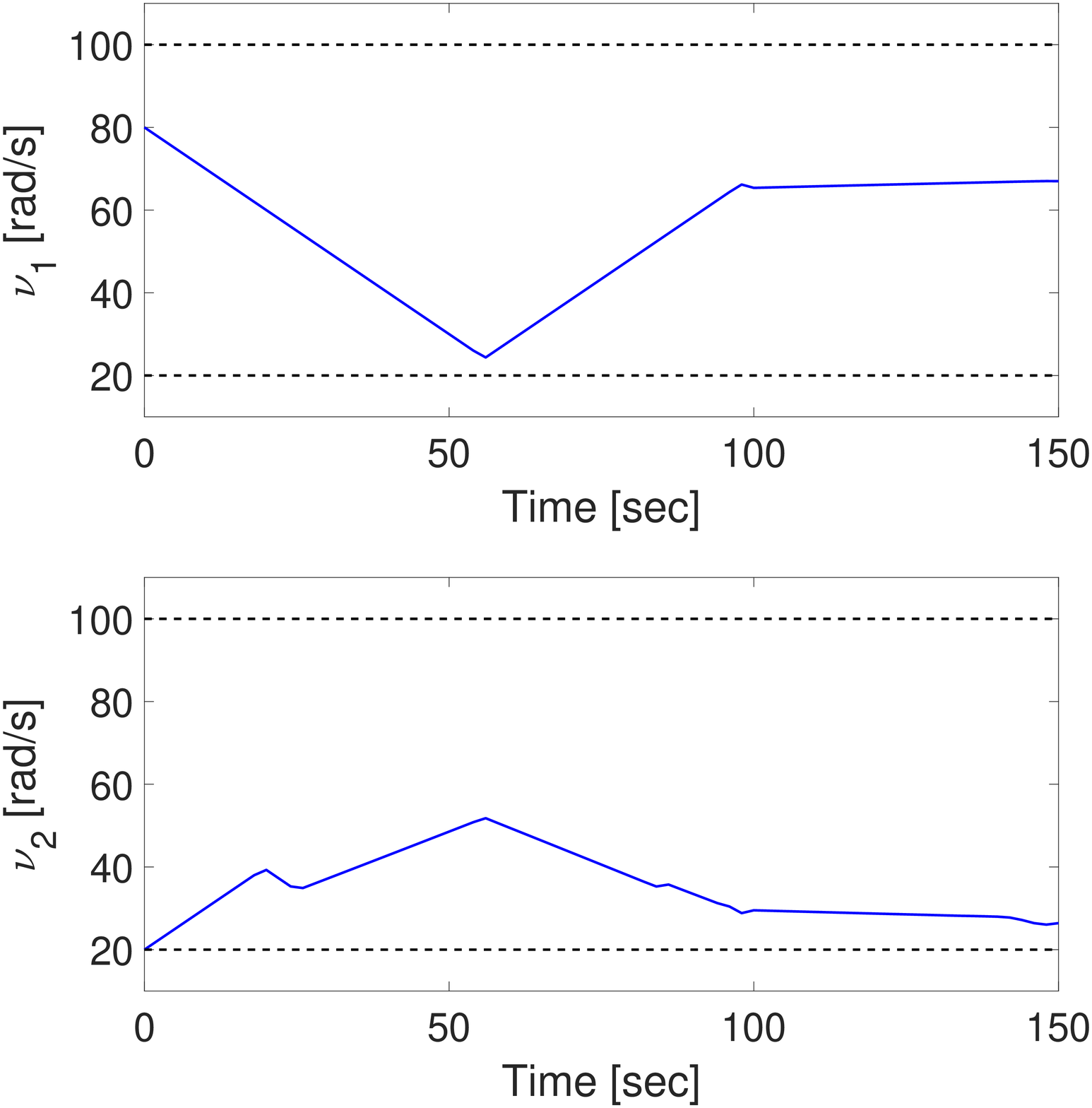,height=2.4in}}  
\put(  0,  0){\epsfig{file=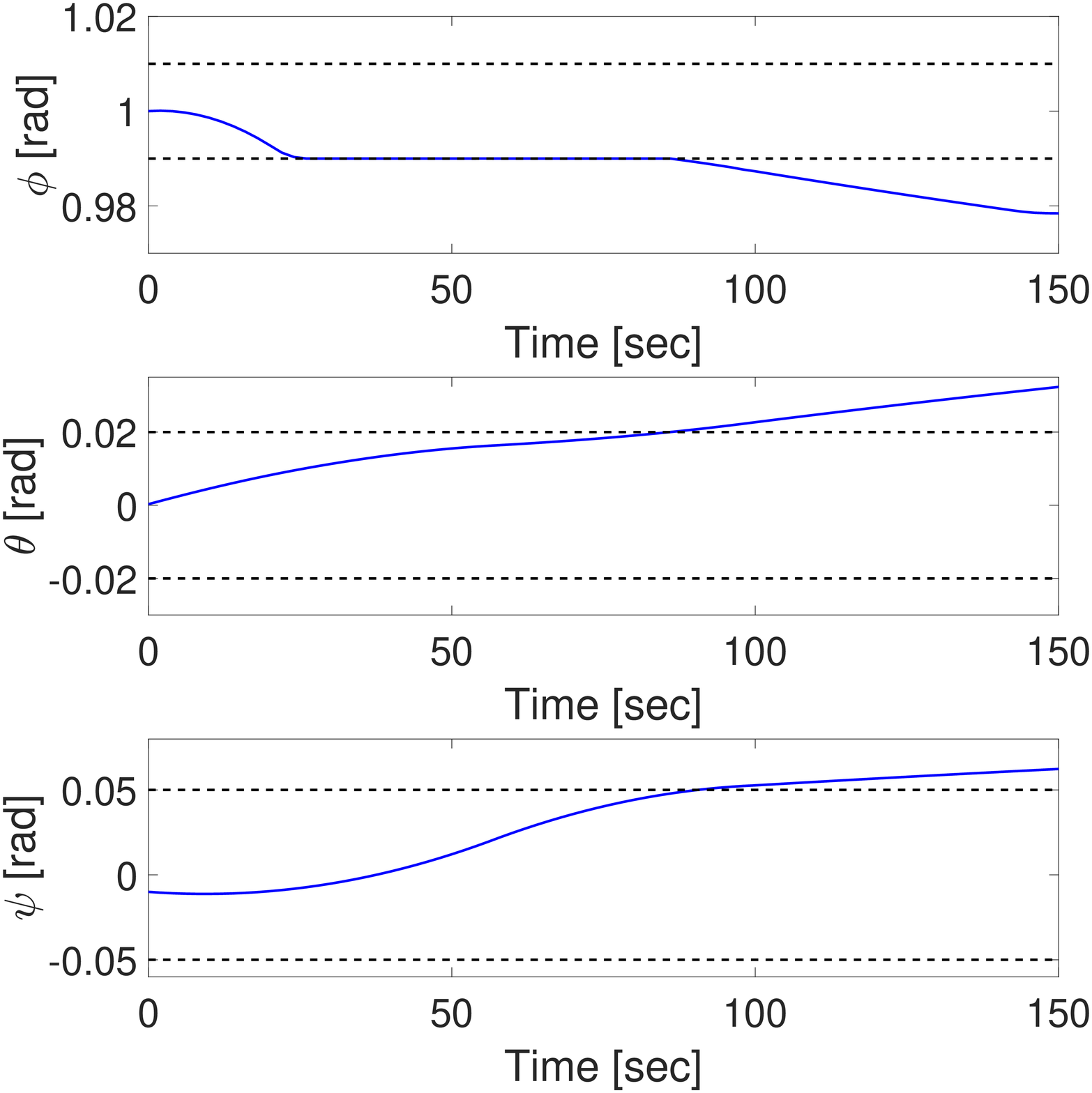,height=2.4in}}  
\put(  200, 0){\epsfig{file=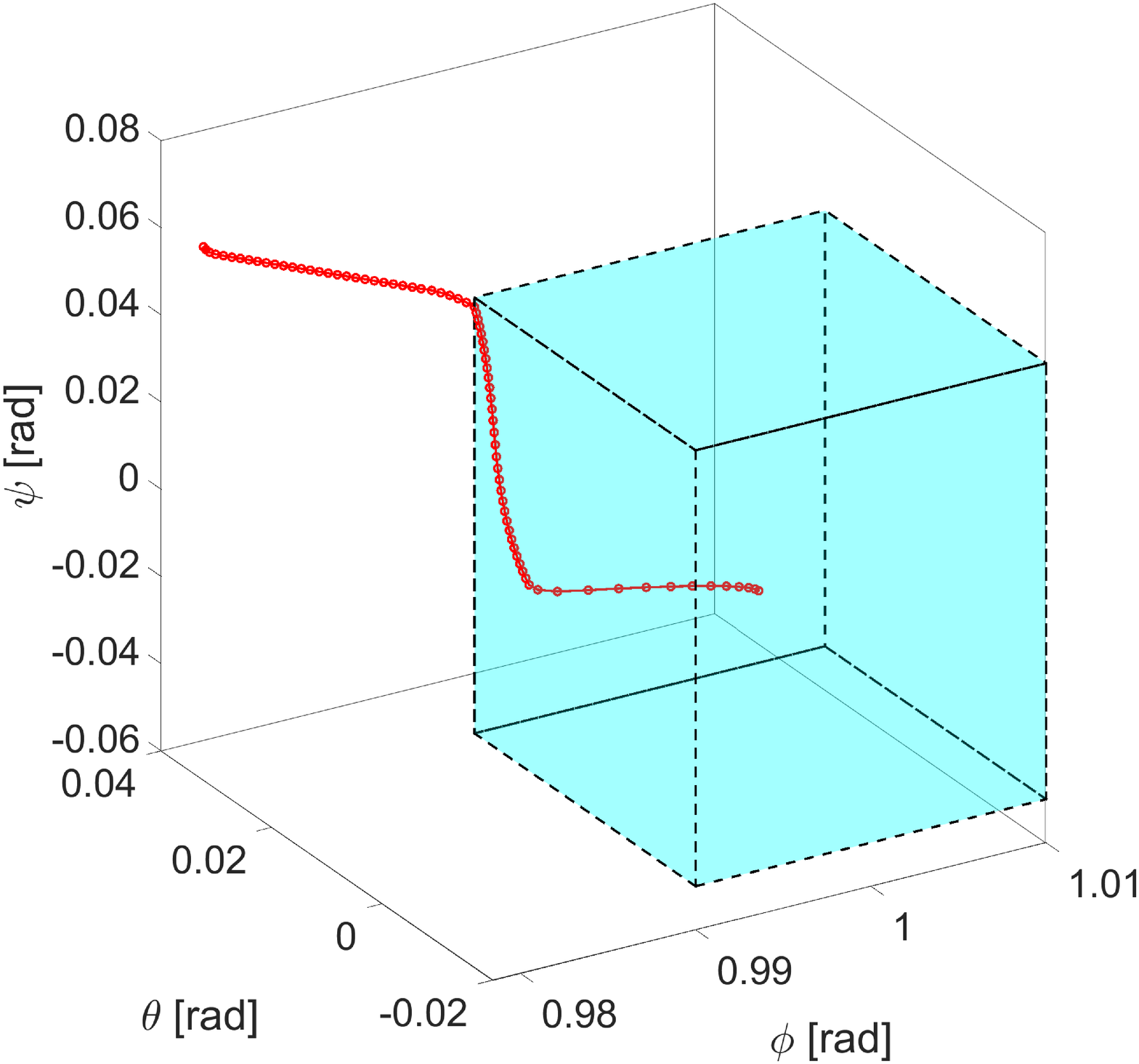,height=2.4in}}
\small
\put(180, 360){(a)}
\put(380, 360){(b)}
\put(180, 180){(c)}
\put(380, 180){(d)}
\normalsize
\end{picture}
\end{center}
      \caption{High precision pointing with two operational RWs (Constraint violation observed).}
      \label{fig:2RW_2}
      \vspace{-0.2in}
\end{figure}

In this example, the control authority is further restricted with smaller admissible RW accelerations. As a result, constraint violations occur at $t = 84\ [sec]$ in the pitch angle $\theta$. Shortly afterwards, at $t = 88\ [sec]$, constraints on raw and yaw angles, $\phi$, $\psi$ are both violated. Although the optimal control cannot prevent Euler angles from drifting out of their desired regions over the prediction horizon, it is observed in (a) that $\phi$ is kept constant at its lower bound between $24$ and $88\ [sec]$, hence postpones the \textit{time-before-exit}, which is expected from a DCOC solution.

\section{Conclusions}\label{sec:5}

This paper presented a continuous optimization approach to DCOC and its application to spacecraft high-precision attitude control. The approach computes a control input sequence that maximizes the \textit{time-before-exit} by solving an NLP problem with an exponentially weighted cost function and purely continuous variables. Based on results from sensitivity analysis and exact penalty method, we proved the optimality guarantee of our approach. The practical application of our approach was demonstrated through a spacecraft high-precision attitude control example. A nominal case with  three functional RWs and an underactuated case with only two functional RWs were considered. Simulation results illustrated the effectiveness of our approach as a contingency method for extending spacecraft's effective mission time in the case of RW failures.

\section*{Funding Sources}

This research was supported by the National Science Foundation award ECCS 1931738 and the Air Force Office of Scientific Research grant FA9550-20-1-038.


\bibliography{ref}

\end{document}